\documentclass[11pt]{amsart}

\usepackage{amsfonts,amsmath,amssymb,amscd}

\usepackage{amsthm}
\usepackage{latexsym}
\usepackage[dvips]{graphicx}
\usepackage[dvips]{psfrag}
\usepackage[dvips]{color}
\usepackage{xypic}
\usepackage[abs]{overpic}
\usepackage{caption}
\usepackage[all]{xy}

\copyrightinfo{2006}{American Mathematical Society}

\newtheorem{theorem}{Theorem}[section]
\newtheorem{lemma}[theorem]{Lemma}

\theoremstyle{definition}

\newtheorem{remark}[theorem]{Remark}

\newcommand{\Nbd}{\operatorname{Nbd}}
\newcommand{\cl}{\operatorname{cl}}

\numberwithin{equation}{section}

\begin{document}

\title[Primitive disk complexes and Goeritz groups]{Connected primitive disk complexes and genus two Goeritz groups of lens spaces}

\author{Sangbum Cho}\thanks{The first-named author is supported in part by Basic Science Research Program through the National Research Foundation of Korea(NRF) funded by the Ministry of Education, Science and Technology (2012006520).}

\address{
Department of Mathematics Education \newline
\indent Hanyang University, Seoul 133-791, Korea}
\email{scho@hanyang.ac.kr}

\author{Yuya Koda}\thanks{The second-named author is supported in part by
Grant-in-Aid for Young Scientists (B) (No. 26800028), Japan Society for the Promotion of Science.}

\address{
Department of Mathematics \newline
\indent Hiroshima University, 1-3-1 Kagamiyama, Higashi-Hiroshima, 739-8526, Japan}
\email{ykoda@hiroshima-u.ac.jp}


\date{\today}

\begin{abstract}
Given a stabilized Heegaard splitting of a $3$-manifold, the primitive disk complex for the splitting is the subcomplex of the disk complex for a handlebody in the splitting spanned by the vertices of the primitive disks. In this work, we study the structure of the primitive disk complex for the genus two Heegaard splitting of each lens space. In particular, we show that the complex for the genus two splitting for the lens space $L(p, q)$ with $1\leq q \leq p/2$ is connected if and only if $p \equiv \pm 1 \pmod q$, and describe the combinatorial structure of each of those complexes. As an application, we obtain a finite presentation of the genus two Goeritz group of each of those lens spaces, the group of isotopy classes of orientation preserving homeomorphisms of the lens space that preserve the genus two Heegaard splitting of it.
\end{abstract}

\maketitle

\section{Introduction}
\label{sec:introduction}

Every closed orientable $3$-manifold can be decomposed into two handlebodies of the same genus, which is called a {\it Heegaard splitting} of the manifold.
The genus of the handlebodies is called the {\it genus} of the splitting.
The $3$-sphere admits a Heegaard splitting of each genus $g \geq 0$, and  lens spaces and $\mathbb S^2 \times \mathbb S^1$ admit Heegaard splittings of each genus $g \geq 1$.

There is a well known simplicial complex, called the {\it disk complex}, for a handlebody and in general for an arbitrary irreducible $3$-manifold with compressible boundary.
The vertices of a disk complex are the isotopy classes of essential disks in the manifold.
When a given a Heegaard splitting is stabilized,
we can define the {\it primitive disk complex} for the splitting, which is the full subcomplex of the disk complex
for a handlebody in the splitting
spanned by the vertices represented by the primitive disks in the handlebody.
Strictly speaking, for each stabilized Heegaard splitting, there are exactly two
primitive disk complexes depending on the choice of a handlebody of the splitting.
However, for all the Heegaard splittings we will consider in this paper, the two primitive disk complexes
are isomorphic. So we simply call it {\it the} primitive disk complex for the splitting.

The first goal of this work is to reveal the combinatorial structure of the primitive disk complex for the genus-$2$ Heegaard splitting of each lens space $L(p, q)$.
For the $3$-sphere and $\mathbb S^2 \times \mathbb S^1$, the structure of the primitive disk complex for the genus-$2$ splitting is well understood from the works \cite{C} and \cite{CK14}. They are both contractible, and further the complex for the $3$-sphere is $2$-dimensional and deformation retracts to a tree in its barycentric subdivision, while the complex for $\mathbb S^2 \times \mathbb S^1$ itself is a tree.
In \cite{C2}, the structure of the primitive disk complex for the genus-$2$ splitting of the lens space $L(p, 1)$ was fully studied.
In addition, a generalized version of a primitive disk complex is also studied in \cite{Kod} for a genus-$2$ handlebody embedded in the $3$-sphere.
In this work, including the case of $L(p, 1)$, we describe the structure of the primitive disk complex for the genus-$2$ splitting in detail for every lens space.
An interesting fact is that not all lens spaces admit connected primitive disk complexes for their genus-$2$ splitting.
In Section \ref{sec:the_structure_of_primitive_disk_complexes}, we find all lens spaces having connected primitive disk complexes for their genus-$2$ splittings (Theorem \ref{thm:contractible}), and then describe the structure of the complex for each lens spaces (Theorem \ref{thm:structure}).

The next goal is to show that the genus-$2$ Goeritz group of the lens space having connected primitive disk complex is finitely presented by giving an explicit presentation of each of them.
Given a Heegaard splitting of a $3$-manifold, the {\it Goeritz group} of the splitting is the group of isotopy classes of orientation preserving homeomorphisms of the manifold that preserve the splitting.
When a genus-$g$ Heegaard splitting for a manifold is unique up to isotopy, we call the Goeritz group of the splitting the {\it genus-$g$ Goeritz group} of the manifold without mentioning a specific splitting of the manifold.
The presentations of those groups have been obtained for some manifolds.
For example, from the works \cite{Go}, \cite{Sc}, \cite{Ak} and \cite{C}, a finite presentation of the genus-$2$ Goeritz group of the $3$-sphere was obtained and from \cite{CK14}, that of $\mathbb S^2 \times \mathbb S^1$ was obtained.
We refer the reader to \cite{Joh10}, \cite{Joh11}, \cite{Sc13}, \cite{CK15}, \cite{CKA}
for finite presentations or finite generating sets of the Goeritz groups of several Heegaard splittings.
For the genus-$2$ Goeriz groups of lens spaces, the finite presentations are obtained only for the lens spaces $L(p, 1)$ in \cite{C2}.
In this work, we show that the genus-$2$ Goeriz group of each lens space having connected primitive disk complex is finitely presented and obtain a presentation of each of them (Theorem \ref{thm:presentations of the Goeritz groups}).
Such a lens space $L(p, q)$ with $1\leq q \leq p/2$ is exactly the one satisfying $p \equiv \pm 1 \pmod q$, which includes the case of $L(p, 1)$.
The basic idea is to investigate the action of the Georitz group on the connected primitive disk complex of each of the lens spaces, and then calculate the isotropy subgroups of its simplices up to the action of the Goeritz group.


We use the standard notation $L = L(p, q)$ for a lens space in standard textbooks.
For example, we refer \cite{Ro} to the reader.
That is, there is a genus one Heegaard splitting of $L$ such that an oriented meridian circle of a solid torus in the splitting is identified with a $(p, q)$-curve on the boundary torus of the other solid torus (fixing oriented longitude and meridian circles of the torus), where $\pi_1(L(p, q))$ is isomorphic to the cyclic group of order $|p|$.
The integer $p$ can be assumed to be positive, and it is well known that two lens spaces $L(p, q)$ and $L(p', q')$ are homeomorphic if and only if $p = p'$ and  $q'q^{\pm 1} \equiv \pm 1 \pmod p$.
Thus we will assume $1 \leq q \leq p/2$ for the lens space $L(p, q)$, or $0 < q < p$ sometimes.
Further, there is a unique integer $q'$ satisfying $1 \leq q' \leq p/2$ and $qq' \equiv \pm 1 \pmod p$, and so, for any other genus one Heegaard splitting of $L(p, q)$, we may assume that an oriented meridian circle of a solid torus of the splitting is identified with a $(p, \bar q)$-curve on the boundary torus of the other solid torus for some $\bar q \in \{q, q', p-q', p-q\}$.

\smallskip

Throughout the paper, $(V, W; \Sigma)$ will denote a genus-$2$ Heegaard splitting of a lens space $L = L(p, q)$.
That is, $V$ and $W$ are genus-$2$ handlebodies such that $V \cup W = L$ and $V \cap W = \partial V = \partial W = \Sigma$ is a genus-$2$ closed orientable surface, which is called a Heegaard surface in $L$.
Any disks in a handlebody are always assumed to be properly embedded, and their intersection is transverse and minimal up to isotopy.
In particular, if a disk $D$ intersects a disk $E$, then $D \cap E$ is a collection of pairwise disjoint arcs that are properly embedded in both $D$ and $E$.
For convenience, we will not distinguish disks (or union of disks) and homeomorphisms from their isotopy
classes in their notation.
Finally, $\Nbd(X)$ will denote a regular neighborhood of $X$ and $\cl(X)$ the closure of $X$ for a subspace $X$ of a polyhedral space, where the ambient space will always be clear from the context.

\section{Primitive disk complexes}
\label{sec:primitive_disk_complexes}
Let $M$ be an irreducible $3$-manifold with compressible boundary.
The {\it disk complex} of $M$ is a simplicial complex defined as follows.
The vertices are the isotopy classes of essential disks in $M$, and a collection of $k+1$ vertices spans a $k$-simplex if and only if it admits a collection of representative disks which are pairwise disjoint.
In particular, if $M$ is a handlebody of genus $g \geq 2$, then the disk complex is $(3g - 4)$-dimensional and is not locally finite.

Let $D$ and $E$ be essential disks in $M$, and suppose that $D$ intersects $E$ transversely and minimally.
Let $C \subset D$ be a disk cut off from $D$ by an outermost arc $\alpha$ of $D \cap E$ in $D$ such that $C \cap E= \alpha$.
We call such a $C$ an {\it outermost subdisk} of $D$ cut off by $D \cap E$.
The arc $\alpha$ cuts $E$ into two disks, say $G$ and $H$.
Then we have two disjoint disks $E_1$ and $E_2$ which are isotopic to disks $G \cup C$ and $H \cup C$ respectively.
We call $E_1$ and $E_2$ the {\it disks from surgery} on $E$ along the outermost subdisk $C$ of $D$.
Since $E$ and $D$ are assumed to intersect minimally, $E_1$ (and $E_2$) is isotopic to neither $E$ nor $D$.
Also at least one of $E_1$ and $E_2$ is non-separating if $D$ is non-separating.
Observe that each of $E_1$ and $E_2$ has fewer arcs of intersection with $D$ than $E$ had since at least the arc $\alpha$ no longer counts.
For an essential disk $D$ in $M$ intersecting  transversely and minimally the union of two disjoint essential disks $E$ and $F$, we define similarly the disks from surgery on $E \cup F$ along an outermost subdisk of $D$ cut off by $D \cap (E \cup F)$.
The following is a key property of a disk complex.

\begin{theorem}
If $\mathcal K$ is a full subcomplex of the disk complex satisfying the following condition, then $\mathcal K$ is contractible.

\begin{itemize}
 \item Let $E$ and $D$ be disks in $M$ representing vertices of $\mathcal K$.
 If they intersect each other transversely and minimally, then at least one of the disks from surgery on $E$ along an outermost subdisk of $D$ cut off by $D \cap E$ represents a vertex of $\mathcal K$.
\end{itemize}
\label{thm:surgery}
\end{theorem}

In \cite{C}, the above theorem is proved in the case where $M$ is a handlebody, but the proof is still valid for an arbitrary irreducible manifold with compressible boundary. From the theorem, we see that the disk complex itself is contractible, and the {\it non-separating disk complex} is also contractible, which is the full subcomplex spanned by the vertices of non-separating disks.
We denote by $\mathcal D(M)$ the non-separating disk complex of $M$.

Consider the case that $M$ is a genus-$2$ handlebody $V$.
Then the complex $\mathcal D(V)$ is $2$-dimensional, and every edge of $\mathcal D(V)$ is contained in infinitely but countably many $2$-simplices.
For any two non-separating disks in $V$ which intersect each other transversely and minimally, it is easy to see that ``both'' of the two disks obtained from surgery on one along an outermost subdisk of another cut off by their intersection are non-separating.
This implies, from Theorem \ref{thm:surgery}, that $\mathcal D(V)$ and the link of any vertex of $\mathcal D(V)$ are all contractible.
Thus the complex $\mathcal D(V)$ deformation retracts to a tree in the barycentric subdivision of it.
Actually, this tree is a dual complex of $\mathcal D(V)$.
A portion of the non-separating disk complex of $V$ together with its dual tree is described  in Figure \ref{disk_complex}.

\begin{center}
\begin{overpic}[width=7cm, clip]{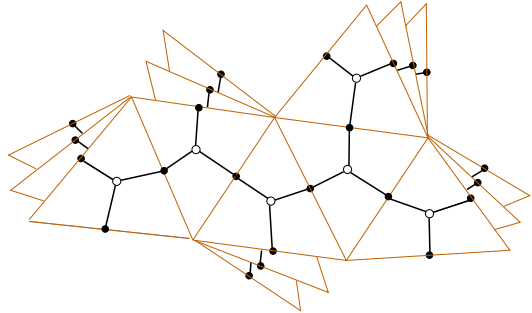}
  \linethickness{3pt}
\end{overpic}
\captionof{figure}{A portion of the non-separating disk complex $\mathcal D(V)$ of a genus-$2$ handlebody $V$ with its dual complex, a tree.}
\label{disk_complex}
\end{center}

Now we return to the genus-$2$ Heegaard splitting $(V, W; \Sigma)$ of a lens space $L = L(p, q)$.
An essential disk $E$ in $V$ is called {\it primitive} if there exists an essential disk $E'$ in $W$ such that $\partial E$ intersects $\partial E' $ transversely in a single point.
Such a disk $E'$ is called a {\it dual disk} of $E$, which is also primitive in $W$ having a dual disk $E$.
Note that both $W \cup \Nbd(E)$ and $V \cup \Nbd(E')$ are solid tori.
Primitive disks are necessarily non-separating.

The {\it primitive disk complex} $\mathcal P(V)$ for the splitting $(V, W; \Sigma)$ is defined to be the full subcomplex of $\mathcal D(V)$ spanned by the vertices of primitive disks in $V$.
From the structure of $\mathcal D(V)$, we observe that every connected component of any full subcomplex of $\mathcal D(V)$ is contractible.
Thus $\mathcal P(V)$ is contractible if it is connected or each of its connected components is contractible otherwise.
In Section \ref{sec:the_structure_of_primitive_disk_complexes}, we describe the complete combinatorial structure of the primitive disk complex $\mathcal P(V)$ for the genus-$2$ Heegaard splitting of each lens space. In particular, we find all lens spaces whose primitive disk complexes for the genus-$2$ splittings are connected, and so contractible.
We first develop several properties of the primitive disks in the following section, which will play a key role throughout the paper.

\section{Primitive disks}
\label{sec:primitive_disks}

\subsection{Primitive elements of the free group of rank two}
\label{subsec:primitive_elements}
The fundamental group of the genus-$2$ handlebody is the free group $\mathbb Z \ast \mathbb Z$ of rank two.
We call an element of $\mathbb Z \ast \mathbb Z$ {\it primitive} if it is a member of a generating pair of $\mathbb Z \ast \mathbb Z$.
Primitive elements of $\mathbb Z \ast \mathbb Z$ have been well understood.
For example, given a generating pair $\{y, z\}$ of $\mathbb Z \ast \mathbb Z$, a cyclically reduced form of any primitive element $w$ can be written as a product of terms each of the form $y^\epsilon z^n$ or $y^\epsilon z^{n+1}$, or else a product of terms each of the form $z^\epsilon y^n$ or $z^\epsilon y^{n+1}$, for some $\epsilon \in \{1,-1\}$ and some $n \in \mathbb Z$.
Consequently, no cyclically reduced form of $w$ in terms of $y$ and $z$ can contain $y$ and $y^{-1}$ $($and $z$ and $z^{-1})$ simultaneously.
Furthermore, we have an explicit characterization of primitive elements containing only positive powers of $y$ and $z$ as follows, which is given in
Osborne-Zieschang \cite{OZ}.

\begin{lemma}
Suppose that $w$ consists of exactly $m$ $z$'s and $n$ $y$'s where $1 \leq m \leq n$.
Then $w$ is primitive if and only if $(m, n) = 1$ and $w$ has the following cyclically reduced form
$$w = w(m, n) = g(1)g(1+m)g(1+2m)\cdots g(1+(m+n-1)m)$$
where the function $g:\mathbb Z \rightarrow \{z, y\}$ is defined by
$$g(i)= g_{m, n}(i) =
\begin{cases}
~z &  \ \text{if~ ~$i \equiv 1, 2, \cdots, m  \pmod{(m+n)}$} \\
~y &  \ \text{otherwise.}
\end{cases}
$$
\label{lem:primitive}
\end{lemma}

For example, $w(3, 5) = zy^2zy^2zy$ and
$w(3, 10) = zy^4zy^3zy^3$.

Let $\{z, y\}$ be a generating pair of the free group of rank two.
Given relatively prime integers $p$ and $q$ with $0 < q <p$, we define a sequence of $(p+1)$ elements $w_0, w_1, \cdots, w_{p-1}, w_p$ in term of $z$ and $w$ as follows.

Define first $w_0$ to be $y^p$.
For each $j \in \{1, 2, \cdots, p\}$, let $f_j:\mathbb Z \rightarrow \{z, y\}$ be the function given by
$$f_j(i)=
\begin{cases}
~z &  \ \text{if ~$i \equiv 1, 1+q, 1+2q, \cdots, 1+(j-1)q \pmod{p}$} \\
~y &  \ \text{otherwise,}
\end{cases}
$$
and then define $w_j = f_j(1)f_j(2)\cdots f_j(p)$.
Each of $w_j$ has length $p$ and consists of $j$ $z$'s and $(p-j)$ $y$'s.
In particular, $w_1 = zy^{p-1}$, $w_{p-1} = z^{p-q}yz^{q-1}$ and $w_p = z^p$.
We call the sequence $w_0, w_1, \cdots w_p$ the {\it $(p, q)$-sequence} of the pair $(z, y)$.
For example, the $(8, 3)$-sequence is given by
\begin{alignat*}{3}
w_0 &= yyyyyyyy &\qquad
w_1 &= zyyyyyyy &\qquad
w_2 &= zyyzyyyy \\
w_3 &= zyyzyyzy &\qquad
w_4 &= zzyzyyzy &\qquad
w_5 &= zzyzzyzy \\
w_6 &= zzyzzyzz &\qquad
w_7 &= zzzzzyzz &\qquad
w_8 &= zzzzzzzz
\end{alignat*}
Observe that $w_{p-j}$ is a cyclic permutation of $\overline{\psi(w_j)}$ for each $j$, where $\psi$ is the automorphism exchanging $z$ and $y$, and $\overline{w}$ is the reverse of $w$.
Thus $w_j$ is primitive if and only if $w_{p-j}$ is primitive.
We can find all primitive elements in the sequence as follows.

\begin{lemma}
Let $w_0, w_1, \cdots, w_p$ be the $(p, q)$-sequence of the generating pair $\{z, y\}$ with $0< q < p$.
Let $q'$ be the unique integer satisfying $1 \leq q' \leq p/2$ with $qq' \equiv \pm 1 \pmod{p}$.
Then $w_j$ is primitive if and only if $j \in \{1, q', p-q', p-1\}$.
\label{lem:four_primitives}
\end{lemma}

\begin{proof}
It is clear that $w_1$ and $w_{p-1}$ are primitive while $w_0$ and $w_p$ are not.

\medskip

{\noindent \sc Claim 1.} $w_{q'}$ is primitive.\\
{\it Proof of Claim 1.}
We write $w_{q'} = f_{q'}(1)f_{q'}(2) \cdots f_{q'}(p)$, and $w(q', p-q') = g(1)g(1+q')g(1+2q')\cdots g(1+(p-1)q')$ where $g = g_{q', p-q'}$ in the notation in Lemma \ref{lem:primitive}.
Since $f(i) = z$ if and only if $i \equiv 1+nq \pmod{p}$ for some $n \in \{0, 1, \cdots, q'-1\}$, it can be directly verified that
$$f_{q'}(i)=
\begin{cases}
 g(1+(i-1)q')&  \ \text{if~ $qq' \equiv 1 \pmod{p}$} \\
 g(1+(i+q)q')&  \ \text{if~ $qq' \equiv -1 \pmod{p}$}.
\end{cases}
$$
Thus $w_{q'}$ is $w(q', p-q')$ itself if $qq' \equiv 1 \pmod{p}$ or is a cyclic permutation of it if $qq' \equiv -1 \pmod{p}$.
In either cases, $w_{q'}$ is primitive.

\medskip

{\noindent \sc Claim 2.} If $1 < j \leq p/2$ and $j \neq q'$, then $w_j$ is not primitive.\\
{\it Proof of Claim 2.}
From the assumption, there is a unique integer $r$ satisfying $2 \leq r \leq p-2$ and $qj \equiv r \pmod{p}$.
Suppose, for contradiction, that $w_j$ is primitive.
Then, by Lemma \ref{lem:primitive}, $(p, j) = 1$ and $w_j$ is a cyclic permutation of $w(j, p-j)$.
We write $w_j = f_j(1) f_j(2) \cdots  f_j(p)$ and $w(j, p-j) = g(1)g(1+j)g(1+2j) \cdots g(1+(p-1)j)$ where $g = g_{j, p-j}$ as in Lemma \ref{lem:primitive}.
Then there is a constant $k$ such that $f_j(i) = g(1+ (i-1+k)j)$ for all $i \in \mathbb Z$.
In particular, $f_j(1 + nq) = z = g(1+(nq + k)j)$ for each $ n \in \{0, 1, \cdots, j-1\}$.

From the definition of $g = g_{j, p-j}$ and the choice of the integer $r$, we have $1+(nq + k)j \equiv 1+nr+kj \equiv 1, 2, \cdots, j \pmod{p}$.
Let $a_n$ be the unique integer satisfying $1+nr + kj \equiv a_n$ with $a_n \in \{1, 2, \cdots, j\}$ for each $n \in \{0, 1, \cdots, j-1\}$.
Observe that $a_n + r \equiv a_{n+1}$ for each $n \in \{0, 1, \cdots, j-2 \}$, and in particular, $a_0 + r \equiv a_1$.
Since $1 \leq a_0 \leq j < p$ and $2 \leq r \leq p-2 < p$, we have only two possibilities: either $a_0 + r = a_1$ or $a_0 + r = a_1 + p$.

First consider the case  $a_0 + r = a_1$.
Then $r \leq j-1$ and $a_n < a_{n+1}$, consequently $a_0 =1, a_1 = 2, \cdots, a_{j-1} = j$, which implies $r=1$, a contradiction.
Next, if $a_0 + r = a_1 + p$, then $p+1-j \leq r$ and $a_n > a_{n+1}$, thus we have $a_0 = j, a_1 = j-1, \cdots, a_{j-1} = 1$, and consequently $r = p-1$, a contradiction again.

\medskip

By the claims, if $1 \leq j \leq p/2$, then $w_j$ is primitive only when $j =1$ or $j=q'$.
If $p/2 \leq j \leq p$, due to the fact that $w_{p-j}$ is a cyclic permutation of $\overline{\psi(w_j)}$, the only primitive elements are $w_{p-q'}$ and $w_{p-1}$, which completes the proof.
\end{proof}

A simple closed curve in the boundary of a genus-$2$ handlebody $W$ represents elements of $\pi_1(W) = \mathbb Z \ast \mathbb Z$.
We call a pair of essential disks in $W$ a {\it complete meridian system} for $W$ if the union of the two disks cuts off $W$ into a $3$-ball.
Given a complete meridian system $\{D, E\}$, assign symbols $x$ and $y$ to the circles $\partial D$ and $\partial E$ respectively.
Suppose that an oriented simple closed curve $l$ on $\partial W$ that meets $\partial D \cup \partial E$ transversely and minimally.
Then $l$ determines a word in terms of $x$ and $y$ which can be read off from the the intersections of $l$ with $\partial D$ and $\partial E$ (after a choice of orientations of $\partial D$ and $\partial E$), and hence $l$ represents an element of the free group $\pi_1 (W) = \left< x, y\right>$.

In this set up, the following is a simple criterion for the primitiveness of the elements represented by such simple closed curves.

\begin{lemma}
With a suitable choice of orientations of $\partial D$ and $\partial E$, if a word corresponding to a simple closed curve $l$ contains one of the pairs of terms$:$ $(1)$  both of $xy$ and $xy^{-1}$ or $(2)$  both of $xy^nx$ and $y^{n+2}$ for $n \geq 0$, then the element of $\pi_1 (W)$ represented by $l$ cannot be $($a positive power of $)$ a primitive element.
\label{lem:key}
\end{lemma}

\begin{center}
\begin{overpic}[width=5cm, clip]{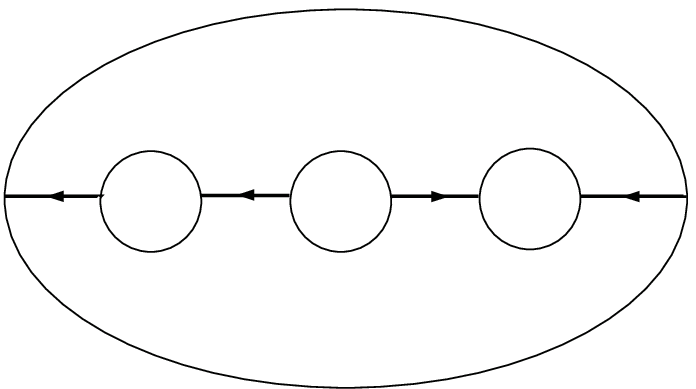}
  \linethickness{3pt}
  \put(8,45){$l_+$}
  \put(43,45){$m_+$}
  \put(82,45){$m_-$}
  \put(126,45){$l_-$}
  \put(25,37){$e_+$}
  \put(65,37){$d_-$}
  \put(104,37){$e_-$}
  \put(65,67){$d_+$}
  \put(65,7){$\Sigma'$}
\end{overpic}
\captionof{figure}{The 4-holed sphere $\Sigma'$.}
\label{not_primitive}
\end{center}

\begin{proof}
Let $\Sigma'$ be the $4$-holed sphere cut off from $\partial W$ along $\partial D \cup \partial E$.
Denote by $d_+$ and $d_-$ (by $e_+$ and $e_-$, respectively) the boundary circles of $\Sigma'$ that came from $\partial D$
(from $\partial E$, respectively).

Suppose first that $l$ represents an element of a form containing both $xy$ and $xy^{-1}$.
Then we may assume that there are two subarcs $l_+$ and $l_-$ of $l \cap \Sigma'$ such that $l_+$ connects $d_+$ and $e_+$, and $l_-$ connects $d_+$ and $e_-$ as in Figure \ref{not_primitive}.
Since $|~l \cap d_+| = |~l \cap d_-|$ and $|~l \cap e_+| = |~l \cap e_-|$, we must have two other arcs $m_+$ and $m_-$ of $l \cap \Sigma'$ such that $m_+$ connects $d_-$ and $e_+$, and $m_-$ connects $d_-$ and $e_-$.
See Figure \ref{not_primitive}.

Consequently, there exists no arc component of $l \cap \Sigma'$ that meets only one of $d_+$, $d_-$, $e_+$ and $e_-$.
That is, any word corresponding to $l$ contains neither $x^{\pm1} x^{\mp1}$ nor $y^{\pm1} y^{\mp1}$, and hence it is cyclically reduced.
Considering all possible directions of the arcs $l_+$, $l_-$, $m_+$ and $m_-$, each word represented by $l$ must contain both $x$ and $x^{-1}$ (or both $y$ and $y^{ -1}$), which means that $l$  cannot represent (a positive power of) a primitive element of $\pi_1 (W)$.

Next, suppose that a word corresponding to $l$ contains $x^2$ and $y^2$, which is the case of $n=0$ in the second condition.
Then there are two arcs $l_+$ and $l_-$ of $l \cap \Sigma'$ such that $l_+$ connects $d_+$ and $d_-$, and $l_-$ connects $e_+$ and $e_-$.
By a similar argument to the above, we see again that any word corresponding to $l$ is cyclically reduced, but contains both of $x^2$ and $y^2$.
Thus $l$ cannot represent (a positive power of) a primitive element.

Suppose that a word corresponding to $l$ contains $xy^nx$ and $y^{n+2}$ for $n \geq 1$.
Then there are two subarcs $\alpha$ and $\beta$ of $l$ which correspond to $xy^nx$ and $y^{n+2}$ respectively.
In particular, we may assume that $\alpha$ starts at $d_+$, intersects $\partial E$ in $n$ points, and ends in $d_-$, while $\beta$ starts at $e_+$, intersects $\partial E$ in its interior in $n$ points, and ends in $e_-$.

Let $m$ be the subarc of $\alpha$ corresponding to $xy$.
Then $m$ connects two circles $d_+$ and one of $e_{\pm}$, say $e_+$.
Choose a disk $E^\ast$ properly embedded in the $3$-ball $W$ cut off by $D \cup E$ such that the boundary circle $\partial E^\ast$ is the frontier of a regular neighborhood of $d_+ \cup m \cup e_+$ in $\Sigma'$.
Then $E^\ast$ is a non-separating disk in $W$ and forms a complete meridian system with $D$.
Assigning the same symbol $y$ to $\partial E^\ast$, the arc $\alpha$ determines $xy^{n-1}x$ while $\beta$ determines $y^{n+1}$.
Thus the conclusion follows by induction.
\end{proof}

\subsection{Primitive disks in a genus-$2$ handlebody}
\label{subsec:primitive_disks}

We recall that $(V, W; \Sigma)$ denotes a genus-$2$ Heegaard splitting of a lens space $L = L(p, q)$.
The primitive disks in $V$ or in $W$ are introduced in Section \ref{sec:primitive_disk_complexes}.
We call a pair of disjoint, non-isotopic primitive disks in $V$ a {\it primitive pair} in $V$.
Similarly, a triple of pairwise disjoint, non-isotopic primitive disks is a {\it primitive triple}.
A non-separating disk $E_0$ properly embedded in $V$ is called {\it semiprimitive} if there is a primitive disk $E'$ in $W$ disjoint from $E_0$.

Any simple closed curve on the boundary of the solid torus $W$ represents an element of $\pi_1 (W)$ which is the free group of rank two.
We interpret primitive disks algebraically as follows, which is a direct consequence of Gordon \cite{Go}.

\begin{lemma}
Let $D$ be a non-separating disk in $V$.
Then $D$ is primitive if and only if $\partial D$ represents a primitive element of $\pi_1 (W)$.
\label{lem:primitive_element}
\end{lemma}

Note that no disk can be both primitive and semiprimitive since the boundary circle of a semiprimitive disk in $V$ represents the $p$-th power of a primitive element of $\pi_1(W)$.

\begin{lemma}
Let $\{D, E\}$ be a primitive pair of $V$.
Then $D$ and $E$ have a common dual disk if and only if there is a semiprimitive disk $E_0$ in $V$ disjoint from $D$ and $E$.
\label{lem:common_dual}
\end{lemma}

\begin{proof}
The necessity is clear.
For sufficiency, let $E'$ be a primitive disk in $W$ disjoint from the semiprimitive disk $E_0$ in $V$.
It is enough to show that $E'$ is a dual disk of every primitive disk in $V$ disjoint from $E_0$, since then $E'$ would be a common dual disk of $D$ and $E$.

\smallskip

\noindent {\it Claim}: If $E$ is a primitive disk in $V$ dual to $E'$, then $E$ is disjoint from $E_0$.

\noindent {\it Proof of claim}.
Denote by $E_0^+$ and $E_0^-$ the two disks on the boundary of the solid torus $V$ cut off by $E_0$ that came from $E_0$.
Suppose that $E$ intersects $E_0$.
Then an outermost subdisk $C$ of $E$ cut off by $E \cap E_0$ must intersect $\partial E'$ since $\partial E'$ is a longitude of the solid torus $V$ cut off by $E_0$.
We may assume that $C$ is incident to $E_0^+$.
Considering $|E \cap E_0^+| = |E \cap E_0^-|$, there is a subarc of $\partial E$ whose two endpoints lie in $\partial E_0^-$, which also intersects $\partial E'$, and hence $\partial E$ intersects $\partial E'$ at least in two points, a contradiction.

\smallskip

Let $D$ be a primitive disk in $V$ disjoint from $E_0$.
Among all the primitive disks in $V$ dual to $E'$, choose one, denoted by $E$ again, such that $|D \cap E|$ is minimal.
By the claim, $E$ is disjoint from $E_0$.
Let $E'_0$ be the unique semiprimitive disk in $W$ disjoint from $E \cup E'$.
Since $\{E', E'_0\}$ forms a complete meridian system of $W$, by assigning symbols $x$ and $y$ to oriented $\partial E'$ and $\partial E'_0$ respectively, any oriented simple closed curve on $\partial W$ represents an element of the free group $\pi_1 (W) = \langle x, y \rangle$ as in the previous section.
In particular, we may assume that $\partial E$ and $\partial E_0$ represents elements of the form $x$ and $y^p$ respectively.

Denote by $\Sigma_0$ the $4$-holed sphere $\partial V$ cut off by $\partial E \cup \partial E_0$.
Consider $\Sigma_0$ as a $2$-holed annulus with two boundary circles $\partial E_0^\pm$ came from $\partial E_0$ and with two holes $\partial E^\pm$ came from $\partial E$.
Then $\partial E'_0$ is the union of $p$ spanning arcs in $\Sigma_0$ which divides $\Sigma_0$ into $p$ rectangles, and the two holes $\partial E^\pm$ is contained in a single rectangle.
Notice that $\partial E'$ is an arc in the rectangle connecting the two holes.
See Figure \ref{Sigma_0} (a).

\begin{center}
\begin{overpic}[width=10.5cm, clip]{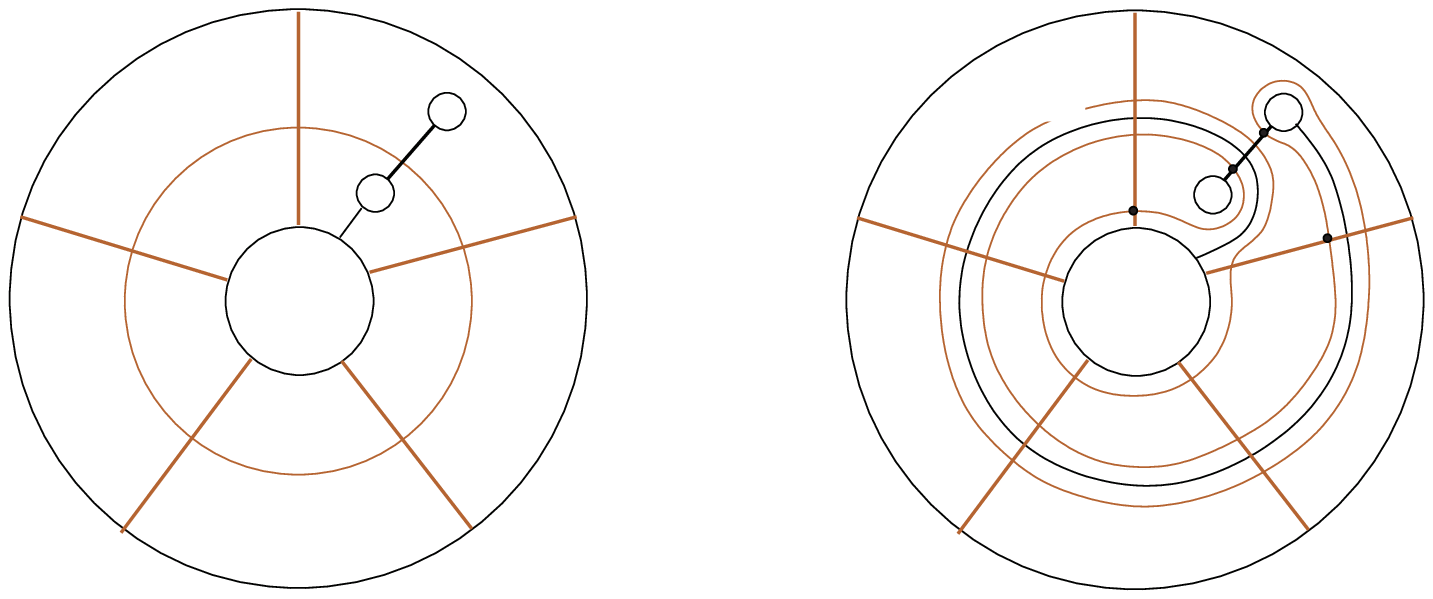}
  \linethickness{3pt}
    \put(55,0){(a)}
    \put(228,0){(b)}
    \put(55,22){$\Sigma_0$}
    \put(228,22){$\Sigma_0$}
    \put(52,72){$\partial E_0^+$}
    \put(225,72){$\partial E_0^+$}
    \put(120,54){$\partial E_0^-$}
    \put(27,106){$\partial D$}
    \put(292,44){$\partial E_0^-$}
    \put(76,117){\footnotesize $\partial E^-$}
    \put(85,93){\footnotesize $\partial E^+$}
    \put(71,105){\footnotesize $\partial E'$}
    \put(67,92){\footnotesize $\alpha$}
    \put(217,113){\footnotesize $\alpha$}
\end{overpic}
\captionof{figure}{The $2$-holed annulus $\Sigma_0$ when $p = 5$, for example.}
\label{Sigma_0}
\end{center}
Suppose that $D$ is disjoint from $E$.
Then $D$ is a non-separating disk in $V$ disjoint from $E \cup E_0$, and hence the boundary circle $\partial D$ can be considered as the frontier of a regular neighborhood in $\Sigma_0$ of the union of one of the two boundary circles, one of the two holes of $\Sigma_0$, and an arc $\alpha$ connecting them.
The arc $\alpha$ cannot intersect $\partial E'_0$ in $\Sigma_0$, otherwise an element represented by $\partial D$ must contain both of $xy$ and $xy^{-1}$ (after changing orientations if necessary), which contradicts that $D$ is primitive by Lemma \ref{lem:key} (see Figure \ref{Sigma_0} (b)).
Thus $\alpha$ is disjoint from $\partial E'_0$, and consequently $D$ intersects $\partial E'$ in a single point.
That is, $E'$ is a dual disk of $D$ (see Figure \ref{Sigma_0} (a)).

Suppose next that $D$ intersects $E$.
Let $C$ be an outermost subdisk of $D$ cut off by $D \cap E$.
Then one of the resulting disks from surgery on $E$ along $C$ is $E_0$ and the other, say $E'$, is isotopic to none of $E$ and $E_0$.
The arc $\partial C \cap \Sigma_0$ can be considered as the frontier of a regular neighborhood of the union of a boundary circle of $\Sigma_0$ came from $\partial E_0$ and an arc, denoted by $\alpha_0$, connecting this circle and a hole came from $\partial E$.
By a similar argument to the above, one can show that $\alpha_0$ is disjoint from $\partial E'_0$, otherwise $D$ would not be primitive.
Consequently, the boundary circle of the resulting disk $E_1$ from the surgery intersects $\partial E'$ in a single point, which means $E_1$ is primitive with the dual disk $E'$.
But we have $|D \cap E_1| < |D \cap E|$ from the surgery construction, which contradicts the minimality of $|D \cap E|$.
\end{proof}

In the proof of Lemma \ref{lem:common_dual}, if we assume that the primitive disk $D$ also intersects $E_0$, then the subdisk $C$ of $D$ cut off by $D \cap (E \cup E_0)$ would be incident to one of $E$ and $E_0$. The argument to show that the resulting disk $E_1$ from the surgery is primitive with the dual disk $E'$ still holds when $C$ is incident to $E_0$ and even when $D$ is semiprimitive. This observation suggests the following lemma.

\begin{lemma}
Let $E_0$ be a semiprimitive disk in $V$ and let $E$ be a primitive disk in $V$ disjoint from $E_0$.
If a primitive or semiprimitive disk $D$ in $V$  intersects $E \cup E_0$, then one of the disks from surgery on $E \cup E_0$ along an outermost subdisk of $D$ cut off by $D \cap (E \cup E_0)$ is either $E$ or $E_0$, and the other, say $E_1$, is a primitive disk, which has a common dual disk with $E$.
\label{lem:first_surgery}
\end{lemma}

\subsection{The link of the vertex of a primitive disk}
\label{subsec:the_link_of_a_primitive_disk}

Again, we have a genus-$2$ Heegaard splitting $(V, W; \Sigma)$ of a lens space $L = L(p, q)$ and we assume $1 \leq q \leq p/2$.
In this section, we introduce a special subcomplex of the non-separating disk complex $\mathcal D(V)$, which we will call a {\it shell} of the vertex of a primitive disk, and then develop its several properties we need.

Let $E$ be a primitive disk in $V$. Choose a dual disk $E'$ of $E$, then we have unique semiprimitive disks $E_0$ and $E'_0$ in $V$ and $W$ respectively which are disjoint from $E \cup E'$.
The circle $\partial E'_0$ is a $(p, \bar{q})$-curve on the boundary of the solid torus $\cl(V- \Nbd(E))$, where $\bar{q} \in \{q, p-q, q', p-q'\}$ and $q'$ is the unique integer satisfying  $1 \leq q' \leq p/2$ and $qq' \equiv \pm 1 \pmod p$.
We first assume that $\partial E'_0$ is a $(p, q)$-curve.
Assigning symbols $x$ and $y$ to oriented $\partial E'$ and $\partial E'_0$ respectively as in the previous sections, any oriented simple closed curve on $\partial W$ represents an element of the free group $\pi_1 (W) = \langle x, y \rangle$.
We simply denote the circles $\partial E'$ and $\partial E'_0$ by $x$ and $y$ respectively.
The circle $y$ is disjoint from $\partial E$ and intersects $\partial E_0$ in $p$ points, and $x$ is disjoint from $\partial E_0$ and intersects $\partial E$ in a single point.
Thus we may assume that $\partial E_0$ and $\partial E$ determine the elements of the form $y^p$ and $x$ respectively.

Let $\Sigma_0$ be the $4$-holed sphere $\partial V$ cut off by $\partial E \cup \partial E_0$.
Denote by $e^\pm$ the boundary circles of $\Sigma_0$ came from $\partial E$ and similarly $e_0^\pm$ came from $\partial E_0$.
The $4$-holed sphere $\Sigma_0$ can be regarded as a $2$-holed annulus where the two boundary circles are $e^\pm_0$ and the two holes $e^\pm$.
Then the circle $y$ in $\Sigma_0$ is the union of $p$ spanning arcs which cuts the annulus into $p$ rectangles, and $x$ is a single arc connecting two holes $e^\pm$, where $x \cup e^\pm$ is contained in a single rectangle (see the surface $\Sigma_0$ in Figure \ref{sequence}).

\begin{center}
\begin{overpic}[width=12.5cm, clip]{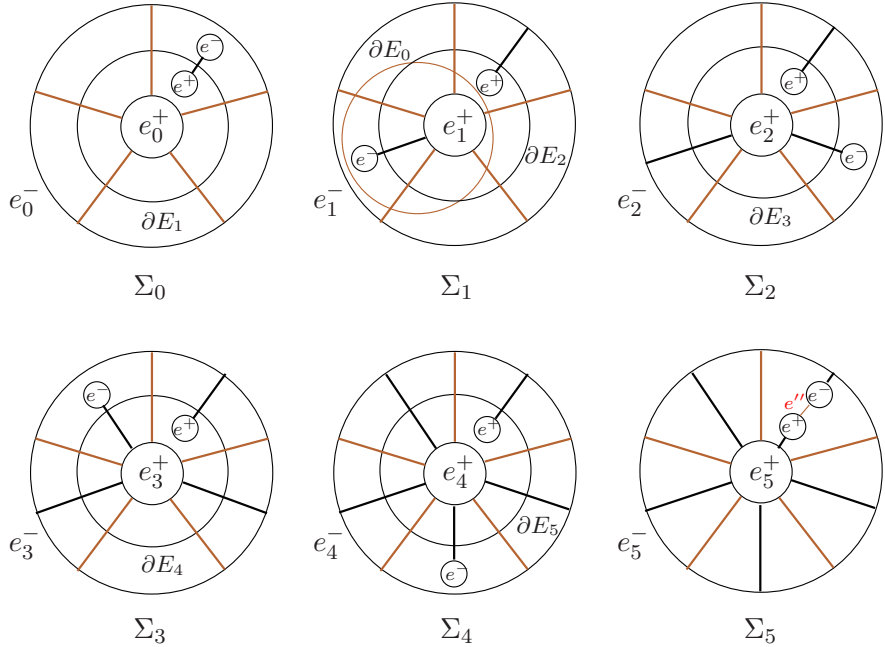}
  \linethickness{3pt}
    \put(57,130){$\Sigma_0$}
    \put(173,130){$\Sigma_1$}
    \put(288,130){$\Sigma_2$}
    \put(57,0){$\Sigma_3$}
    \put(173,0){$\Sigma_4$}
    \put(288,0){$\Sigma_5$}
    \put(59,191){$e_0^+$}
    \put(174,191){$e_1^+$}
    \put(290,191){$e_2^+$}
    \put(59,60){$e_3^+$}
    \put(174,60){$e_4^+$}
    \put(290,60){$e_5^+$}
    \put(10,163){$e_0^-$}
    \put(125,163){$e_1^-$}
    \put(240,163){$e_2^-$}
    \put(10,33){$e_3^-$}
    \put(125,33){$e_4^-$}
    \put(240,33){$e_5^-$}
    \put(72,207){\tiny $e^+$}
    \put(188,207){\tiny $e^+$}
    \put(303,208){\tiny $e^+$}
    \put(72,77){\tiny $e^+$}
    \put(187,77){\tiny $e^+$}
    \put(302,78){\tiny $e^+$}
    \put(82,222){\tiny $e^-$}
    \put(141,179){\tiny $e^-$}
    \put(326,180){\tiny $e^-$}
    \put(39,90){\tiny $e^-$}
    \put(175,22){\tiny $e^-$}
    \put(312,90){\tiny $e^-$}
    \put(60,155){\footnotesize $\partial E_1$}
    \put(146,220){\footnotesize $\partial E_0$}
    \put(205,180){\footnotesize $\partial E_2$}
    \put(290,157){\footnotesize $\partial E_3$}
    \put(60,25){\footnotesize $\partial E_4$}
    \put(202,40){\footnotesize $\partial E_5$}
    \put(303,87){\tiny \color{red} $e''$}
\end{overpic}
\captionof{figure}{The disks in a $(5, 2)$-shell in $\mathcal D(V)$ for $L(5, 2)$.}
\label{sequence}
\end{center}

Any non-separating disk in $V$ disjoint from $E \cup E_0$ and not isotopic to either of $E$ and $E_0$ is determined by an arc properly embedded in $\Sigma_0$ connecting one of $e^\pm$ and one of $e^\pm_0$.
That is, the boundary circle of such a disk is the frontier of a regular neighborhood of the union of the arc and the two circles connected by the arc in $\Sigma_0$.
Choose such an arc $\alpha_0$ so that $\alpha_0$ is disjoint from $y$, and denote by $E_1$ the non-separating disk determined by $\alpha_0$.
Observe that there are infinitely many choices of such arcs $\alpha_0$ up to isotopy, and so are the disks $E_1$.
But the element represented by $\partial E_1$ has one of the forms $x^{\pm 1} y^{\pm p}$, so we may assume that $\partial E_1$ represents $xy^p$ by changing the orientations if necessary.

Next, let $\Sigma_1$ be the $4$-holed sphere $\partial V$ cut off by $\partial E \cup \partial E_1$.
As in the case of $\Sigma_0$, consider $\Sigma_1$ as a $2$-holed annulus with boundaries $e^\pm_1$ and with two holes $e^\pm$ where $e^\pm_1$ came from $\partial E_1$.
Then the circle $y$ cuts off $\Sigma_1$ into $p$ rectangles as in the case of $\Sigma_0$, but two holes $e^+$ and $e^-$ are now contained in different rectangles.
In particular, we can give labels $0,1,\ldots, p-1$ to the rectangles consecutively so that $e^+$ lies in the $0$-th rectangle while $e^-$ in the $q$-th rectangle.
The circle $x$ in $\Sigma_1$ is the union of two arcs connecting $e^\pm_1$ and $e^\pm$ contained in the $0$-th and $p$-th rectangles respectively.

Now consider a properly embedded arc in $\Sigma_1$ connecting one of $e^\pm$ and one of $e_1^\pm$.
Choose such an arc $\alpha_1$ so that $\alpha_1$ is disjoint from $y$ and parallel to none of the two arcs of $x \cap \Sigma_1$.
Then $\alpha_1$ determines a non-separating disk, denoted by $E_2$, whose boundary circle is the frontier of a regular neighborhood of the union of $\alpha_1$ and the two circles connected by $\alpha_1$.
(If $\alpha_1$ is isotopic to one of the two arcs $x \cap \Sigma_1$, then the resulting disk is $E_0$.)
Observe that $\partial E_2$ represents an element of the form $xy^q xy^{p-q}$ (see the surface $\Sigma_1$ in Figure \ref{sequence}).

We continue this process in the same way.
Then $\Sigma_2$ is the $4$-holed sphere $\partial V$ cut off by $\partial E \cup \partial E_2$, and we choose an arc $\alpha_2$ in $\Sigma_2$ disjoint from $y$ and parallel to none of the arcs $x \cap \Sigma_2$, which determines the disk $E_3$.
The boundary circle $\partial E_3$ represents an element of the form $xy^qxy^qxy^{p-2q}$.
In general, we have a non-separating disk $E_j$ whose boundary circle lies in the $4$-holed sphere $\Sigma_{j-1}$.
We finish the process in the $p$-th step to have the disk $E_p$ whose boundary circle lies in $\Sigma_{p-1}$.
The disk $E_{p-1}$ and $E_p$ represent elements of the form $(xy)^{p-q}y(xy)^{q-1}$ and $(xy)^p$ respectively.
Observe that there are infinitely many choices of the arc $\alpha_0$, and so choices of the disk $E_1$ as we have seen, but once $E_1$ have been chosen, the next disks $E_j$ for each $j \in \{1, 2, \cdots, p-1\}$ are uniquely determined.

\begin{center}
\begin{overpic}[width=6cm, clip]{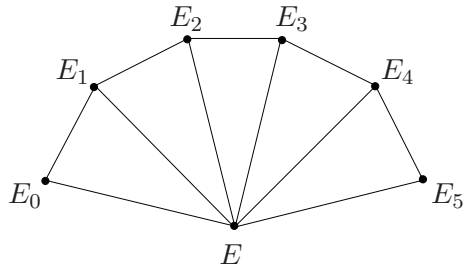}
  \linethickness{3pt}
  \put(80,0){$E$}
  \put(0,23){$E_0$}
  \put(18,70){$E_1$}
  \put(61,90){$E_2$}
  \put(101,90){$E_3$}
  \put(141,70){$E_4$}
  \put(160,23){$E_5$}
\end{overpic}
\captionof{figure}{A $(5, 2)$-shell.}
\label{fig:shell}
\end{center}

We call the full subcomplex of $\mathcal D(V)$ spanned by the vertices $E_0, E_1, \cdots, E_p$ and $E$ a {\it shell} centered at the primitive disk $E$ and denote it simply by $\mathcal S_E = \{E_0, E_1, \cdots, E_p\}$.
In particular, since the circle $\partial E'_0$ is assumed to be a $(p, q)$-curve in the beginning of the construction, the shell $\mathcal S_E$ is called a $(p, q)$-shell.
In general, given a genus-$2$ splitting of the lens space $L(p, q)$, we might have $(p, \bar{q})$-shell by the same construction, where $\bar{q} \in \{q, p-q, q', p-q'\}$ and $q'$ is the unique integer satisfying  $1 \leq q' \leq p/2$ and $qq' \equiv \pm 1 \pmod p$.
We observe that there exist infinitely many shells centered at any primitive disk $E$ by the choice of a dual disk $E'$.
Further there exist infinitely many shells centered at $E$ containing the vertex of a semiprimitive disk $E_0$ disjoint from $E$. That is, there are infinitely many choices of the primitive disks $E_1$ disjoint from $E \cup E_0$.
On the contrary, once the disk $E_1$ is chosen, the shell centered at $E$ and containing $E_0$ and $E_1$ is uniquely determined.
Figure \ref{fig:shell} illustrates a $(5, 2)$-shell in $\mathcal D(V)$ in the splitting of $L(5, 2)$.

\begin{remark}
For any consecutive vertices $E_j$, $E_{j+1}$ and $E_{j+2}$ in a shell $\mathcal S_E = \{E_0$, $E_1, \cdots, E_p\}$, the disk $E_j$ is disjoint from $E_{j+1}$, and intersects $E_{j+2}$ in a single arc for each $j \in \{0, 1, \cdots, p-2\}$.
For example, see $\partial E_0$, $\partial E_2$ and $\partial E_1$ ($= e_1^{\pm}$) in $\Sigma_1$ in Figure \ref{sequence}.
In general, we have $|E_i \cap E_j| = j - i -1$ for $0 \leq i < j \leq p$.
This is obvious from the construction.
Figure \ref{intersection} illustrates intersections of $E_j$ with $E_{j+2}$, $E_{j+3}$ and $E_{j+4}$ in the $3$-balls $V$ cut off by $E \cup E_{j+1}$, $E \cup E_{j+2}$ and $E \cup E_{j+3}$ respectively.
\label{remark:intersections_in_a_shell}
\end{remark}

\begin{center}
\begin{overpic}[width=12.5cm, clip]{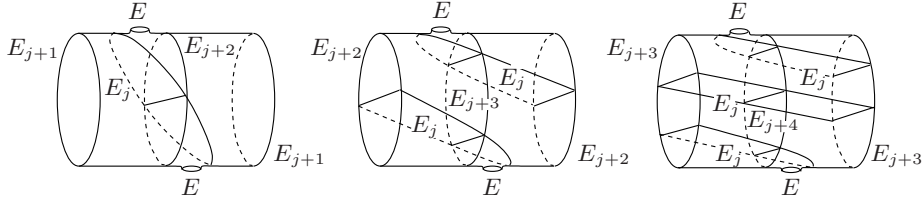}
  \linethickness{3pt}
  \put(70,2){\footnotesize $E$}
  \put(50,69){\footnotesize $E$}
  \put(4,55){\footnotesize $E_{j+1}$}
  \put(105,15){\footnotesize $E_{j+1}$}
  \put(71,56){\footnotesize $E_{j+2}$}
  \put(41,40){\footnotesize $E_j$}

  \put(184,2){\footnotesize $E$}
  \put(165,69){\footnotesize $E$}
  \put(119,55){\footnotesize $E_{j+2}$}
  \put(220,15){\footnotesize $E_{j+2}$}
  \put(171,36){\footnotesize $E_{j+3}$}
  \put(190,44){\footnotesize $E_j$}
  \put(157,26){\footnotesize $E_j$}

  \put(297,2){\footnotesize $E$}
  \put(277,69){\footnotesize $E$}
  \put(231,55){\footnotesize $E_{j+3}$}
  \put(331,15){\footnotesize $E_{j+3}$}
  \put(284,28){\footnotesize $E_{j+4}$}
  \put(304,44){\footnotesize $E_j$}
  \put(271,34){\footnotesize $E_j$}
  \put(272,15){\footnotesize $E_j$}
\end{overpic}
\captionof{figure}{Intersections of $E_j$ with $E_{j+2}$, $E_{j+3}$ and $E_{j+4}$.}
\label{intersection}
\end{center}

\begin{lemma}
Let $\mathcal S_E = \{E_0, E_1, \cdots, E_{p-1}, E_p\}$ be a $(p, q)$-shell
centered at a primitive disk $E$ in $V$.
Then we have
\begin{enumerate}
\item $E_0$ and $E_p$ are semiprimitive.
\item $E_j$ is primitive if and only if $j \in \{1, q', p-q', p-1\}$ where $q'$ is the unique integer satisfying $qq' \equiv \pm1 \pmod p$ and $1 \leq q' \leq p/2$.
\end{enumerate}
\label{lem:sequence}
\end{lemma}

\begin{proof}
(1) $E_0$ is a semiprimitive disk disjoint from $E'$ from the construction.
For the disk $E_p$, it is easy to find a circle $e''$ in $\Sigma$ such that $e'' \cap \Sigma_p$ is an arc which connects the two holes $e^+$ and $e^-$ and is disjoint from $x \cup y \cup e^+_p \cup e^-_p$ (see the arc $e''$ in the surface $\Sigma_5$ in Figure \ref{sequence}).
Cutting $W$ along $E' \cup E'_0$, we have a $3$-ball $B$, and the circle $e''$ lies in $\partial B$.
Thus $e''$ bounds a disk $E''$ in $W$ which is primitive since $e''$ intersects $\partial E$ in a single point.
The disk $E_p$ is disjoint from $E''$ and so is semiprimitive.

(2) From the construction, each circle $\partial E_j$ represents the element $w_j$ in the $(p, q)$-sequence in section \ref{subsec:primitive_elements}, by the substitution of $z$ for $xy$.
Thus the conclusion follows by Lemma \ref{lem:four_primitives} with Lemma \ref{lem:primitive_element}.
\end{proof}

\begin{remark}
We have constructed a $(p, q)$-shell $\mathcal S_E$ by assuming $\partial E'_0$ is a $(p, q)$-curve in the beginning of the construction.
If $\mathcal S_E$ is a $(p, p-q)$-shell, then we have the same conclusion of Lemma \ref{lem:sequence}. If $\mathcal S_E$ is a $(p, q')$-shell or a $(p, p-q')$-shell, the Lemma \ref{lem:sequence} still holds by exchanging $q$ and $q'$ in the conclusion.
Also, we observe that a $(p, q)$-shell $\mathcal S_E = \{E_0, E_1, \cdots, E_{p-1}, E_p\}$ is identified with the $(p, p-q)$-shell $\mathcal S'_E = \{E_p, E_{p-1}, \cdots, E_1, E_0\}$ centered at the same $E$ if we choose the dual disk $E''$ of $E$ and then choose the primitive disk $E_{p-1}$ disjoint from $E \cup E_p$.
\end{remark}

The following is a generalization of Lemma \ref{lem:first_surgery}.

\begin{lemma}
Let $\mathcal S_E = \{E_0, E_1, \cdots, E_{p-1}, E_p\}$ be a shell centered at a primitive disk $E$ in $V$, and let $D$ be a primitive or semiprimitive disk in $V$.
For $j \in \{1, 2, \cdots, p-1\}$,
\begin{enumerate}
\item if $D$ is disjoint from $E \cup E_j$ and is isotopic to none of $E$ and $E_j$, then $D$ is isotopic to either $E_{j-1}$ or $E_{j+1}$, and
\item if $D$ intersects $E \cup E_j$, then one of the disks from surgery on $E \cup E_j$ along an outermost subdisk $C$ of $D$ cut off by $D \cap (E \cup E_j)$ is either $E$ or $E_j$, and the other is
either $E_{j-1}$ or $E_{j+1}$.
\end{enumerate}
\label{lem:surgery_on_primitive}
\end{lemma}

\begin{proof}
Suppose that $D$ is disjoint from $E \cup E_j$.
The boundary circle $\partial D$ lies in the $2$-holed annulus $\Sigma_j$.
Thus $\partial D$ can be considered as the frontier of the union of one hole and one boundary circle of $\Sigma_j$, and an arc $\alpha_j$ connecting them.
By the same argument for the proof of Lemmas \ref{lem:common_dual} and \ref{lem:first_surgery}, the arc $\alpha_j$ cannot intersect the arcs of $\partial E'_0 \cap \Sigma_j$ otherwise $D$ would not be (semi)primitive.
Thus the disk $D$ must be either $E_{j-1}$ or $E_{j+1}$.
(Note that if both of $E_{j-1}$ and $E_{j+1}$ are not primitive, then we can say that such a primitive disk $D$ does not exist.)
The second statement can be proved in the same manner by considering the arc $\partial C \cap \Sigma_j$  for the outermost subdisk $C$ of $D$.
\end{proof}

\subsection{Primitive disks intersecting each other}
\label{subsec:primitive_disks_intersecting_each_other}

The following is the main theorem of this section.

\begin{theorem}
Given a lens space $L(p, q)$, $1 \leq q \leq p/2$, with a genus-$2$ Heegaard splitting $(V, W; \Sigma)$, suppose that $p \equiv \pm 1 \pmod{q}$.
Let $D$ and $E$ be primitive disks in $V$ which intersect each other transversely and minimally.
Then at least one of the two disks from surgery on $E$ along an outermost subdisk of $D$ cut off by $D \cap E$ is primitive.\par
\label{thm:primitive}
\end{theorem}

\begin{proof}
Let $C$ be an outermost subdisk of $D$ cut off by $D \cap E$.
The choice of a dual disk $E'$ of $E$ determines a unique semiprimitive disk $E_0$ in $V$ which is disjoint from $E \cup E'$.
Among all the dual disks of $E$, choose one, denoted by $E'$ again, so that the resulting semiprimitive disk $E_0$ intersects $C$ minimally.
If $C$ is disjoint from $E_0$, then, by Lemma \ref{lem:first_surgery}, the disk from surgery on $E$ along $C$ other than $E_0$ is primitive, having the common dual disk $E'$ with $E$, and so we are done.

From now on, we assume that $C$ intersects $E_0$.
Then one of the disks from surgery on $E_0$ along an outermost subdisk $C_0$ of $C$ cut off by $C \cap E_0$ is $E$, and the other, say $E_1$, is primitive having the common dual disk $E'$ with $E$, by Lemma \ref{lem:first_surgery} again.
Then we have the shell $\mathcal S_E = \{E_0, E_1, E_2, \cdots, E_p\}$ centered at $E$.
Let $E'_0$ be the unique semiprimitive disk in $W$ disjoint from $E \cup E'$.
The circle $\partial E'_0$ would be a $(p, \bar q)$-curve on the boundary of the solid torus $\cl(V- \Nbd(E \cup E'))$ for some $\bar q \in \{q, q', p-q', p-q\}$, where $q'$ satisfies $1 \leq q' \leq p/2$ and $qq' \equiv \pm 1 \pmod{p}$.
We will consider only the case of $\bar q = q$.
That is, $\partial E'_0$ is a $(p, q)$-curve and so $\mathcal S_E$ is a $(p, q)$-shell.
The proof is easily adapted for the other cases.

If $C$ intersects $E_1$, then one of the disks from surgery on $E_1$ along an outermost subdisk $C_1$ of $C$ cut off by $C \cap E_1$ is $E$,
and the other is either $E_0$ or $E_2$ by Lemma \ref{lem:surgery_on_primitive}, but it is actually $E_2$ since we have $|C \cap E_1| < |C \cap E_0|$ from the surgery construction.
In general, if $C$ intersects each of $E_1, E_2, \cdots, E_j$, for $j \in \{1, 2, \cdots, p-1\}$, the disk from surgery on $E_j$ by an outermost subdisk $C_j$ of $C$ cut off by $C \cap E_j$, other than $E$, is $E_{j+1}$, and we have $|C \cap E_{j+1}| < |C \cap E_j|$.
Consequently, we see that $|C \cap E_p| < |C \cap E_0|$, but it contradicts the minimality of $|C \cap E_0|$ since $E_p$ is also a semiprimitive disk disjoint from $E$.
Thus, there is a disk $E_j$ for some $j \in \{1, 2, \cdots, p-1\}$ which is disjoint from $C$.

Now, denote by $E_j$ again the first disk in the sequence that is disjoint from $C$.
Then the two disks from surgery on $E$ along $C$ are $E_j$ and $E_{j+1}$, hence $C$ is also disjoint from $E_{j+1}$.
Actually they are the only disks in the sequence disjoint from $C$.
For other disks in the sequence, it is easy to see that $|C \cap E_{j-k}| = k = |C \cap E_{j+1+k}|$ (by a similar observation to the fact that $|E_i \cap E_j| = j - i -1$ for $0 \leq i < j \leq p$ in Remark \ref{remark:intersections_in_a_shell}).
If $j \geq p/2$, then we have $|C \cap E_0| = j > p-j-1 = |C \cap E_p|$, a contradiction for the minimality condition again.
Thus, $E_j$ is one of the disks in the first half of the sequence, that is, $1 \leq j < p/2$.

\smallskip

{\noindent \sc Claim.}
The disk $E_j$ is one of $E_1$, $E_{q'-1}$ or $E_{q'}$, where $q'$ is the unique integer satisfying $1 \leq q' \leq p/2$ and $qq' \equiv \pm 1 \pmod{p}$.

\smallskip

{\noindent \it Proof of Claim.}
We have assumed that $p \equiv \pm 1 \pmod{q}$ with $1 \leq q \leq p/2$, and so $q' = 1$ if $q = 1$, and $p= qq' + 1$ if $q=2$, and $p = qq' \pm 1$ if $q \geq 3$.
Assigning symbols $x$ and $y$ to oriented $\partial E'$ and $\partial E_0'$ respectively, $\partial E_{q'}$ may represent the primitive element of the form $xy^qxy^q \cdots xy^qxy^{q \pm 1}$ if $q \geq 2$ or $xy^p$ if $q = 1$.
In general, $\partial E_k$ represents an element of the form $xy^{n_1}xy^{n_2} \cdots xy^{n_k}$ for some positive integers $n_1, \cdots, n_k$ with $n_1+ \cdots + n_k = p$ for each $k \in \{1, 2, \cdots, p\}$.
Furthermore, since $C$ is disjoint from $E_j$ and $E_{j+1}$, the word determined by the arc $\partial C \cap \Sigma_j$ is of the form $y^{m_1}xy^{m_2} \cdots xy^{m_{j+1}}$ (or its reverse) when $\partial E_{j+1}$ represents an element of the form $xy^{m_1}xy^{m_2} \cdots xy^{m_{j+1}}$.

If $2 \leq j \leq q'-2$, then an element represented by $\partial E_{j+1}$ has the form $xy^qxy^q \cdots xy^q xy^{p-jq}$, and so an element represented by $\partial D$ contains $xy^qx$ and $y^{p-jq}$, which lies in the part $\partial C \cap \Sigma_j$ of $\partial D$.
We have $q' \geq 4$ in this case, and so $q \geq 2$.
Thus $p-jq = qq' \pm 1 - jq \geq q+2$.
By Lemma \ref{lem:key}, the disk $D$ cannot be primitive, a contradiction.

Suppose that $q' < j <p/2$.
First, observe that $\partial E_{q'+1}$ represents an element of the form $xy^q \cdots xy^qxy$ if $p = qq'+1$ or $xyxy^{q-1}xy^q \cdots xy^qxy^{q-1}$ if $p=qq' -1$.
Also a word represented by $\partial E_{j+1}$ is obtained by changing one $xy^q$ of a word represented by $\partial E_j$ into $xy^{q-1}xy$ or $xyxy^{q-1}$.
Thus, when we write $xy^{n_1}xy^{n_2} \cdots xy^{n_{j+1}}$ a word represented by $\partial E_{j+1}$, at least one of $n_2, n_3, \cdots, n_j$ must be $1$, and one of $n_1, n_2, \cdots, n_{j+1}$ is greater than $2$.
Since $C$ is disjoint from $E_j$ and $E_{j+1}$, the word corresponding to $\partial C \cap \Sigma_j$ is of the form $y^{n_1}xy^{n_2} \cdots xy^{n_{j+1}}$, which contains both of $xyx$ and $y^n$ for some $n > 2$.
Consequently, by Lemma \ref{lem:key}, the disk $D$ cannot be primitive, a contradiction again.

\smallskip

From the claim, at least one of the disks from surgery on $E$ along $C$ is either $E_1$ or $E_{q'}$.
The disk $E_1$ is primitive, and since we assumed that the circle $\partial E'_0$ is a $(p, q)$-curve on the boundary of the solid torus $\cl(V - \Nbd(E \cup E'))$, the disk $E_{q'}$ is also primitive by Lemma \ref{lem:sequence}, which completes the proof.
\end{proof}

In the proof of the above theorem, we assumed $\bar q = q$, which implied that a resulting disk from surgery is $E_1$ or $E_{q'}$.
The same result holds when $\bar q = p-q$.
But if we assume $\bar q \in \{q', p-q'\}$, then the resulting disk will be $E_1$ or $E_q$ which are primitive.
Together with this observation, assuming that $D$ is disjoint from $E$, and so taking the disk $D$ instead of an outermost subdisk $C$ in the proof of Therorem \ref{thm:primitive}, we have the following result.


\begin{lemma}
Given a lens space $L(p, q)$, $0 < q < p$, with a genus-$2$ Heegaard splitting $(V, W; \Sigma)$, let $\{E, D\}$ be a primitive pair of $V$.
Then there exists a unique shell $\mathcal S_E = \{E_0, E_1, \cdots, E_p\}$ centered at $E$ containing $D$. That is, $D$ is one of $E_0, E_1, \cdots, E_p$.
Furthermore, if $\mathcal S_E$ is a $(p, q)$-shell, then the vertex $D$ is one of $E_1$, $E_{q'}$, $E_{p-q'}$ or $E_{p-1}$, where $q'$ is the unique integer satisfying $1 \leq  q' \leq p/2$ and $qq' \equiv \pm 1 \pmod{p}$.
\label{lem:shell}
\end{lemma}

Let $D$ be an essential disk in $V$.
We denote by $V_D$ the solid torus $\cl(V - \Nbd(D))$.
We remark that $V_D$ and its exterior form a genus-$1$
Heegaard splitting of $L(p,q)$ if and only if
$D$ is a primitive disk in $V$.
We refine the above lemma as follows.

\begin{lemma}
Given a lens space $L(p, q)$, $0 < q < p$, with a genus-$2$ Heegaard splitting $(V, W; \Sigma)$, let $\{E, D\}$ be a primitive pair of $V$.
Let $\mathcal S_E = \{E_0, E_1, \cdots, E_p\}$ and $\mathcal S_D = \{D_0, D_1, \cdots, D_p\}$ be the unique shells centered at $E$ and at $D$ containing $D$ and $E$ respectively.
Assume further that $\mathcal S_E$ is a $(p,q)$-shell.
\begin{enumerate}
\item
\label{item:dual shells for common dual}
If $\{E, D\}$ has a common dual disk, then $\mathcal S_D$ is a $(p,q)$-shell.
Further, $E$ is $D_1$ or $D_{p-1}$ and $D$ is $E_1$ or $E_{p-1}$.
\item
\label{item:dual shells for non-common dual}
If $\{E, D\}$ has no common dual disk, then $\mathcal S_D$ is a $(p,q')$-shell, where
$q'$ is the unique integer with $1 \leq  q' \leq p/2$ and $qq' \equiv \pm 1 \pmod{p}$.
Further,
$D$ is $E_{q'}$ or $E_{p-q'}$ and $E$ is $D_q$ or $D_{p-q}$.
\end{enumerate}
\label{lem:shells_crossing}
\end{lemma}

\begin{proof}
Let $E'$ ($D'$, respectively)
be the unique dual disks of $E$ ($D$, respectively) disjoint from $E_0$
($D_0$, respectively),
and let $E_0'$ ($D_0'$, respectively)
be the unique semi-primitive disk in $W$ disjoint from
$E$ ($D$, respectively).

\noindent (\ref{item:dual shells for common dual})
Suppose $\{ E, D \}$ has a common dual disk.
Then $V_D$ is isotopic to $V_E$ in $L(p,q)$.
This implies that $\partial D'_0$ is also a $(p,q)$-curve on $\partial V_D$.
Hence $\mathcal S_D$ is a $(p,q)$-shell as well.
It is clear that  $E$ is $D_1$ or $D_{p-1}$ and $D$ is $E_1$ or $E_{p-1}$ by Lemma \ref{lem:common_dual}.


\noindent (\ref{item:dual shells for non-common dual})
Suppose $\{ E, D \}$ has no common dual disk.
We note that and $1 < q < p-1$ in this case, and so $1 < q' \leq p/2$.
By Lemma \ref{lem:common_dual} and Lemma \ref{lem:shell},
$D$ is one of $E_{q'}$ and $E_{p-q'}$, and
$E$ is one of $D_{q}$, $D_{q'}$, $D_{p-q}$ and $D_{p-q'}$.

The solid torus $V_D$ and its exterior form
a genus-1 Heegaard splitting of $L(p,q)$.
We will show that $V_E$ is not isotopic to the solid torus $V_D$.
Let $E'$ be a dual disk of $E$ that has minimal intersection with
$D$.
Let $l_D$ and $l_E$ be the core loops of the solid tori
$V_D$ and $V_E$, respectively.
We may assume that $l_D$ and $l_E$ intersect
$E$ and $D$, respectively, once and transversely.
See Figure \ref{fig:core_loops} (a).

\begin{center}
\begin{overpic}[width=10cm, clip]{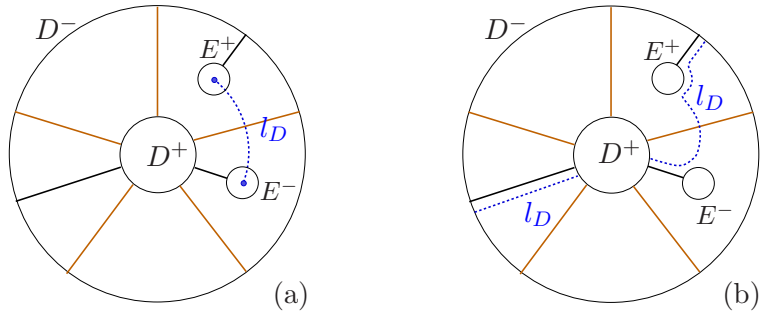}
  \linethickness{3pt}
  \put(100,0){(a)}
  \put(52,52){$D^+$}
  \put(10,100){$D^-$}
  \put(72,93){\small $E^+$}
  \put(95,39){\small $E^-$}
  \put(95,63){\color{blue} $l_D$}

  \put(270,0){(b)}
  \put(223,53){$D^+$}
  \put(180,100){$D^-$}
  \put(240, 92){\small $E^+$}
  \put(260,30){\small $E^-$}
  \put(195, 31){\color{blue} $l_D$}
  \put(260,75){\color{blue} $l_D$}
\end{overpic}
\captionof{figure}{The loop $l_D$ in the case of $L(5,2)$.}
\label{fig:core_loops}
\end{center}
We may move $l_D$ by isotopy in $V \cup \Nbd(E')$ so that
$l_D$ lies in $\partial V_E$.
See Figure \ref{fig:core_loops}  (b).
Now the two core circles $l_E$ and $l_D$ lie in the solid torus $V_E$ of which $D$ is a meridian disk. We observe that the circle $l_D$ intersects $D$ in $q'$ points transversely and minimally after an isotopy, while the circle $l_E$ intersects $D$ in a single point.
That is, we see that $[l_D] = q'[l_E]$ in $H_1(L(p,q))$ after giving
a suitable orientation on each of $l_D$ and $l_E$.
Since $1 < q' \leq p/2$, this implies that $V_D$ and $V_E$ are not isotopic in $L(p,q)$.
By the uniqueness of a genus-1 Heegaard surface of $L(p,q)$,
$V_E$ is actually isotopic to the solid torus which is the exterior of $V_D$.
This implies that $\partial D'_0$ is a $(p,q')$-curve on $\partial V_D$.
Thus $\mathcal S_D$ is a $(p,q')$-shell, and hence $E$ is $D_{q}$ or $D_{p-q}$.
\end{proof}

\begin{remark}
If we assume that $\mathcal S_E$ is a $(p,q')$-shell
instead of a $(p,q)$-shell in Lemmas \ref{lem:shell} and \ref{lem:shells_crossing},
the conclusion is obtained by
replacing $q'$ by $q$ and vice versa.
\end{remark}

\section{The structure of primitive disk complexes}
\label{sec:the_structure_of_primitive_disk_complexes}

\subsection{Contractibility theorem}
\label{subsec:connectivity_theorem}

The goal of this section is to find all lens spaces whose primitive disk complexes for the genus-$2$ splittings are connected and so contractible, Theorem \ref{thm:contractible}.
As in the previous sections, let $E$ be a primitive disk in $V$ with a dual disk $E'$.
The disk $E'$ forms a complete meridian system of $W$ together with the semiprimitive disk $E'_0$ in $W$ disjoint from $E \cup E'$.
Assigning the symbols $x$ and $y$ to the oriented circles $\partial E'$ and $\partial E'_0$ respectively, any oriented simple closed curve, especially the boundary circle of any essential disk in $V$, represents an element of the free group $\pi_1(W) = \langle x, y \rangle$ in terms of $x$ and $y$.
Let $D$ be a non-separating disk in $V$.
A simple closed curve $l$ on $\partial V$
intersecting $\partial D$ transversely in a
single point is called a {\it dual circle} of $D$.
We say that $l$ is a {\it common dual circle} of two disks if it is a dual circle of each of the disks.
We start with the following lemma.

\begin{lemma}
\label{lem:key lemma for non-connectivity}
Let $\{D_1, D_2\}$ be a complete meridian system of $V$.
Suppose that the non-separating disks $D_1$ and $D_2$ satisfy the following conditions:
\begin{enumerate}
\item
for each $i \in \{ 1,2 \}$, all intersections of $\partial D_i$
and $\partial E^\prime$ have the same sign;
\item
for each $i \in \{ 1 , 2 \}$, the circle $\partial D_i$ represents an element $w_i$ of the form
$(xy^q)^{m_i} x y^{n_i}$, where
$0 \leq m_1 < m_2$ and $n_1 \neq n_2$;
\item
any subarc of $\partial E^\prime$ with both endpoints on $\partial D_1$
intersects $\partial D_2$; and
\item
there exists a common dual circle $l$ of $D_1$ and $D_2$ on $\partial V$ disjoint from
$\partial E^\prime$.
\end{enumerate}
Then there exists a non-separating disk $D_*$ in $V$
disjoint from $D_1 \cup D_2$ satisfying the following:
\begin{enumerate}
\item
all intersections of $\partial D_*$
and $\partial E^\prime$ have the same sign;
\item
$\partial D_*$ represents an element of the form
$(xy^q)^{m_1 + m_2 + 1} x y^{n_1 + n_2 - q}$;
\item
for each $i \in \{1 , 2 \}$, any subarc of
$\partial E^\prime$ with both endpoints on $\partial D_i$
intersects $\partial D_*$; and
\item
for each $i \in \{1 , 2 \}$,
there exists a common dual circle of $D_i$ and $D_*$ on $\partial V$ disjoint from
$\partial E^\prime$.
\end{enumerate}
\end{lemma}

\begin{proof}
For $i \in \{1,2\}$, let $\nu_i$
be a connected subarc of
$\partial D_i$
that determines the subword
$y^{n_i}$ of $w_i$.
Cutting off $\partial V$ by $\partial D_1 \cup \partial D_2$,
we obtain the 4-holed sphere $\Sigma_\ast$.
We denote by $d_i^\pm$ the boundary circles
of $\Sigma_\ast$ coming from $\partial D_i$, and
by $\nu_i^\pm$ the subarc of $d_i^\pm$
coming from $\nu_i$.
By the assumption (2),
we may assume without loss of generality that
each oriented arc component
$\partial E^\prime \cap \Sigma_\ast$
directs from $d_{i_1}^+$ to $d_{i_2}^-$ for certain
$i_1, i_2 \in \{ 1 , 2 \}$.
By the assumptions (3) and (4),
the 4-holed sphere $\Sigma_\ast$ and the arcs $\Sigma_\ast \cap \partial E^\prime$ and
$\Sigma_\ast \cap l = l' \sqcup l''$
on $\Sigma_\ast$ can be drawn as in one of
Figure \ref{fig:L_or_R-replacement} (a) and (b).
\begin{center}
\begin{overpic}[width=12cm, clip]{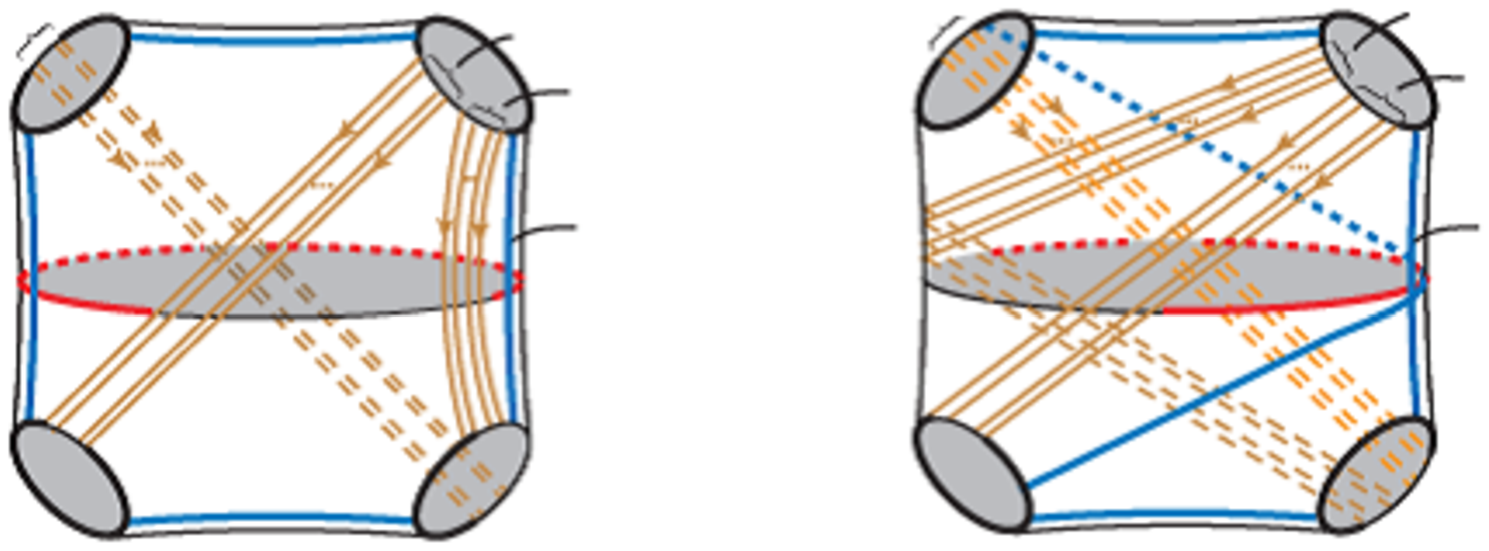}
  \linethickness{3pt}
  \put(55,0){(a)}
  \put(56,15){$\Sigma_*$}
  \put(58,128){\color{blue} $l'$}
  \put(58,36){\color{blue} $l''$}
  \put(15,100){\color{blue} $l_1^*$}
  \put(133,94){\color{blue} $l_2^*$}
  \put(120,137){$m_1+1$}
  \put(131,124){$m_2 - m_1$}
  \put(-22,140){$m_1+1$}
  \put(-1,83){\color{red} $\nu_*$}
  \put(122,83){\color{red} $\nu_*'$}
  \put(85,82){$D_*$}
  \put(18,147){$\nu_1^+$}
  \put(10,23){$\nu_1^-$}
  \put(105,145){$\nu_2^+$}
  \put(100,22){$\nu_2^-$}
  \put(16,129){\small $d_1^+$}
  \put(16,39){\small $d_1^-$}
  \put(106,126){\small $d_2^+$}
  \put(103,40){\small $d_2^-$}

  \put(255,0){(b)}
  \put(256,15){$\Sigma_*$}
  \put(258,129){\color{blue} $l$}
  \put(58,36){\color{blue} $l'$}
  \put(258,37){\color{blue} $l'$}
  \put(238,52){\color{blue} $l_1^*$}
  \put(330,94){\color{blue} $l_2^*$}
  \put(314,141){$m_2 - m_1$}
  \put(326,127){$m_1 +1$}
  \put(176,140){$m_1+1$}
  \put(235,97){\color{red} $\nu_*'$}
  \put(319,84){\color{red} $\nu_*$}
  \put(285,83){$D_*$}
  \put(216,148){$\nu_1^+$}
  \put(207,25){$\nu_1^-$}
  \put(298,148){$\nu_2^+$}
  \put(299,23){$\nu_2^-$}
  \put(213,129){\small $d_1^+$}
  \put(214,40){\small $d_1^-$}
  \put(300,129){\small $d_2^+$}
  \put(300,40){\small $d_2^-$}
\end{overpic}
\captionof{figure}{The 4-holed sphere $\Sigma_*$. There are 2 patterns of $\partial E' \cap \Sigma_*$.}
\label{fig:L_or_R-replacement}
\end{center}
In the figure the arcs $\nu_i^{\pm}$ in $d^{\pm}_i$ are
drawn in bold.

Let $D_*$ be the horizontal disk shown in each of
Figure \ref{fig:L_or_R-replacement} (a) and (b).
It is clear that
$D_*$ satisfies conditions (1) and (3).
For each $i \in \{ 1 , 2 \}$
the simple closed curve on $\partial V$
obtained from the arc
$l_i^*$ depicted in the figure
by gluing back along $d_1^{\pm}$ and $d_2^\pm$
is a common dual circle of $D_i$ and $D_*$ disjoint from
$E^\prime$, hence the condition (4) holds.
Moreover, it is easily seen that
all but one component $\nu_*$ of
$\partial D_*$ cut off by $\partial E^\prime$,
shown in Figure \ref{fig:L_or_R-replacement},
determine a word of the form $y^q$.
Hence it suffices to show that
the arc $\nu_*$ determines a word of the
form $y^{n_1 + n_2 - q}$.
From the arcs $\nu_1^+ \cup \nu_2^+$, algebraically
$n_1 + n_2$ arcs of $\partial E_0^\prime \cap \Sigma_\ast$
come down and all of them pass trough
$\nu_* \cup \nu_*^\prime$ from above,
where the arc $\nu_*^\prime$ is shown in
Figure \ref{fig:L_or_R-replacement}.
Since the arc $\nu_*^\prime$ determines a word of the
form $y^q$,
the arc $\nu_*$ determines a word of the
form $y^{n_1 + n_2 - q}$.
\end{proof}

Let $(D_1, D_2)$ be an ordered pair of
disjoint non-separating disks in $V$
such that the (unordered) pair $\{D_1, D_2 \}$
satisfies the conditions of Lemma
\ref{lem:key lemma for non-connectivity}.
Then there exists a disk $D_*$
as in the lemma
and we again obtain new ordered pairs
$(D_1, D_*)$ and $(D_*, D_2)$ such that
both $\{ D_1, D_* \}$ and $\{ D_*, D_2 \}$
satisty the conditions of the lemma.
We call these new pairs $(D_1 , D_*)$ and
$(D_* , D_2)$ the pairs obtained by
{\it R-replacement} and {\it L-replacement}, respectively,
of $(D_1, D_2)$.

\vspace{1em}

\begin{theorem}
For a lens space $L(p, q)$ with $1 \leq q \leq p/2$,
the primitive disk complex $\mathcal P(V)$ for
a genus-$2$ Heegaard splitting $(V, W; \Sigma)$ of $L(p, q)$ is contractible if and only if $p \equiv \pm 1 \pmod{q}$.
\label{thm:contractible}
\end{theorem}

\begin{proof}
The ``if" part follows directly from Theorem \ref{thm:primitive} and Theorem \ref{thm:surgery}.
For the ``only if" part, we will show that $\mathcal P(V)$ is not connected
when $p \not\equiv \pm 1 \pmod{q}$.
Suppose that $p \not\equiv \pm 1 \pmod{q}$.
Let $m$ and $r$ be integers such that $p=qm + r$ with $2 \leq r \leq q-2$.
Then there exist a natural number $s$ and a non-negative integer $t$ with
$s r - (t+1) q = 1$.
Consider the unique continued fraction expansion
\[
s / (t+1) =
p_0 +  \frac{1}{p_1 + \frac{1}{p_2 + \frac{1}{\ddots + \frac{1}{p_k}}}} =: [ p_0 ; p_1 , p_2 , \ldots , p_k ] \]
where $p_i \geq 1$ for $i \in \{0, 1, \cdots, k-1\}$ and $p_k \geq 2$.

The circle $\partial E_0'$ is a $(p, \bar q)$-curve on
the boundary of the solid torus $V_E$
for some $\bar q \in \{q, q', p-q', p-q\}$,
where $q'$ satisfies $1 \leq q' \leq p/2$ and $qq' \equiv \pm 1 \pmod p$.
We will consider only the case of $\bar q = q$, that is,
$\partial E_0'$ is a $(p, q)$-curve on
the boundary of $V_E$.
The following argument can be easily adapted for the other cases.

Consider any $(p, q)$-shell $\mathcal S_E = \{E_0, E_1, \cdots, E_p\}$ in $\mathcal D(V)$ centered at $E$.
Note that the disks $E_m$ and $E_{m+1}$ in the sequence are not primitive since $\partial E_m$ and $\partial E_{m+1}$ represent elements of the form $(xy^q)^{m-1}xy^{q + r}$ and $(xy^q)^m xy^r$ respectively.
Set $D_0 = E_m$ and $D_{-1} = E$.
Since $D_0$ and $D_{-1}$ satisfy the conditions of
Lemma \ref{lem:key lemma for non-connectivity},
we obtain a new ordered pair $(D_0 , D_1)$ by an
R-replacement of $(D_0, D_{-1})$.
The disk $D_1$ is not primitive since $\partial D_1$ represents an element of the form
$(xy^q)^m xy^r$.
(Actually, $D_1$ can be chosen to be the disk $E_{m+1}$ in the sequence.)
Applying R-replacements $(p_0 - 1)$ times more, starting at
$(D_0 , D_1)$, as
\[ (D_0, D_1) \to (D_0, D_2) \to \cdots  \to (D_0, D_{p_0}) , \]
we obtain the pair $(D_0 , D_{p_0})$.
Next we apply L-replacements $p_1$ times starting at $(D_0 , D_{p_0})$ as
\[ (D_0, D_{p_0}) \to (D_{p_0 + 1}, D_{p_0}) \to (D_{p_0 + 2}, D_{p_0})
\to \cdots  \to (D_{p_0 + p_1}, D_{p_0}) \]
to obtain the pair $(D_{p_0 + p_1}, D_{p_0})$.
Continuing this process, we finally obtain either the pair
$(D_{p_0 + \cdots + p_k}, D_{p_0 + \cdots + p_{k-1}})$
if $k$ is odd, or
$(D_{p_0 + \cdots + p_{k-1}}, D_{p_0 + \cdots + p_k})$
if $k$ is even, of pairwise disjoint non-separating disks.
See Figure \ref{fig:sequence_of_triangles}.

\begin{center}
\begin{overpic}[width=12.5cm, clip]{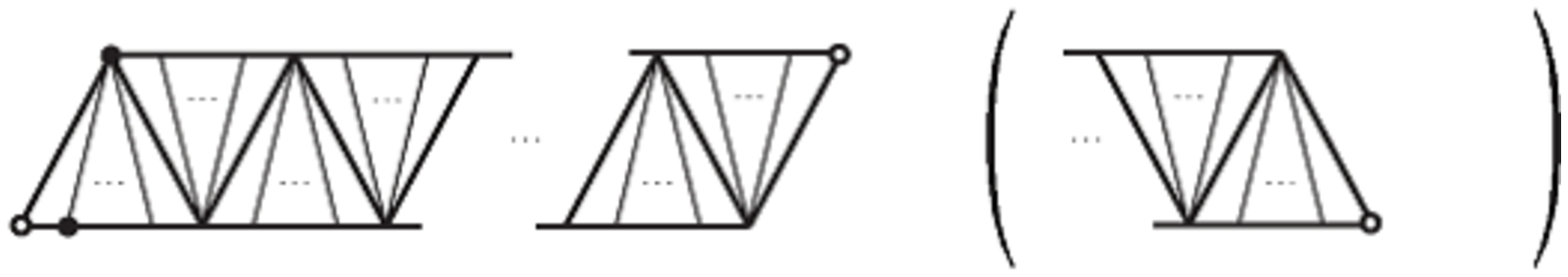}
  \linethickness{3pt}
  \put(-15,5){\tiny $E = D_{-1}$}
  \put(20,5){\tiny $D_1$}
  \put(50,5){\tiny $D_{p_0}$}
  \put(90,5){\tiny $D_{p_0+p_1+p_2}$}
  \put(172,5){\tiny $D_{p_0+ \cdots p_{k-1}}$}
  \put(7,56){\tiny $E_m=D_0$}
  \put(72,56){\tiny $D_{p_1 + p_2}$}
  \put(192,56){\tiny $D_{p_0+ \cdots +p_k}$}
  \put(233,30){or}
  \put(290,56){\tiny $D_{p_0+ \cdots p_{k-1}}$}
  \put(309,5){\tiny $D_{p_0+ \cdots +p_k}$}
\end{overpic}
\captionof{figure}{The portion of $\mathcal{D}(V)$ obtained by
L and R-replacements from $(D_0, D_{-1})$ following the process that corresponds to the continued fraction $[ p_0 ; p_1 , p_2 , \ldots , p_k ]$.
The vertices $D_{-1}$ and $D_{p_0+ \cdots +p_k}$ are primitive whereas $D_0$ and $D_1$ are not primitive.}
\label{fig:sequence_of_triangles}
\end{center}

We assign $D_0$ and $D_{-1}$ the rational numbers
$1/0$ and $0/1$, respectively.
We inductively assign rational numbers to the disks appearing in the above process as follows.
Let $(D_* , D_{**})$ be an ordered pair of non-separating disks
appearing in the process.
Assume that we have already assigned $D_*$ and $D_{**}$
rational numbers $a_1 / b_1$ and $a_2 / b_2$, respectively.
Then we assign the next disk obtained by
L or R-replacement of $(D_*, D_{**})$
the rational number $(a_1 + a_2) / (b_1 + b_2)$.

\smallskip

\noindent {\it Claim.}
If a disk $D_j$, for $-1 \leq j \leq p_0 + \cdots + p_k$,
appearing in the above process is assigned
a rational number $a / b$, then $\partial D_j$
represents an element of the form
$(xy^q)^{d} xy^{a r - (b-1)q}$ for some non-negative integer $d$.

\smallskip

\noindent {\it Proof of Claim.}
If $j = -1$, then $a / b = 0 / 1$ and
$\partial D_{-1} = \partial E$ represents
$x$ and we have $a r - (b - 1)q = 0$.
If $j = 0$, then $a / b = 1 / 0$ and
$\partial D_{0} = \partial E_m$ represents an element of the form
$(xy^q)^{m - 1}xy^{q+r}$ and we have $a r - (b - 1)q = q + r$.

Assume that the claim is true for any $D_i$ with
$i$ less than $j$ and that $D_j$ is obtained from $(D_* , D_{**})$.
If $D_*$ and $D_{**}$
are assigned rational numbers $a_1 / b_1$ and $a_2 / b_2$,
respectively, then $D_j$ is assigned $(a_1 + a_2) / (b_1 + b_2)$ by definition.
By the assumption, $\partial D_*$ and $\partial D_{**}$
determine elements of the forms
$(xy^q)^{d_1} xy^{a_1 r - (b_1-1)q}$ and
$(xy^q)^{d_2} xy^{a_2 r - (b_2-1)q}$ respectively,
for some non-negative integers $d_1$ and $d_2$.
By Lemma \ref{lem:key lemma for non-connectivity},
the circle $\partial D_j$ determines an element of the form
$(xy^q)^{d_1 + d_2 + 1} xy^{(a_1+a_2) r - (b_1+b_2-1)q}$, and hence the induction completes the proof.

\smallskip

Due to well-known properties of the Farey graph, see e.g. Hatcher-Thurston \cite{HT85},
$D_{p_0 + \cdots + p_k}$ is assigned $s / (t+1)$.
Therefore, by the claim,
$\partial D_{p_0 + \cdots + p_k}$
determines an element of the form
$(xy^q)^d xy^{sr - tq}$, hence
$(xy^q)^d xy^{q+1}$.
This implies that $D_{p_0 + \cdots + p_k}$ is primitive.

Now, we focus on the four disks
$D_{-1}$, $D_0$, $D_1$ and $D_{p_0 + \cdots + p_k}$.
Since the dual complex of the disk complex
$\mathcal{D}(V)$ is a tree, and
the disks $D_0$ and $D_1$ are not primitive,
the primitive disks $D_{-1}$ and $D_{p_0 + \cdots + p_k}$
belong to different components of $\mathcal{P}(V)$.
This implies that $\mathcal{P}(V)$ is not connected.
\end{proof}

\subsection{The structures of primitive disk complexes}
\label{subsec:connected_primitive_disk_complexes}
In this section, we describe the combinatorial structure of the primitive disk complex for the genus-$2$ Heegaard splitting of each lens space.
We say simply that a primitive pair has a common dual disk if the two disks of the pair have a common dual disk.

\begin{theorem}
Given a lens space $L(p, q)$, $1 \leq q \leq p/2$, with a genus-$2$ Heegaard splitting $(V, W; \Sigma)$, each primitive pair in $V$ has a common dual disk if and only if $q = 1$.
In this case, if $p \geq 3$, the pair has a unique common dual disk, and if $p = 2$, the pair has  exactly two disjoint common dual disks, which form a primitive pair in $W$.
\label{thm:common_dual}
\end{theorem}

\begin{proof}
Suppose that $q = 1$, and let $\{D, E\}$ be any primitive pair of $V$.
By Lemma \ref{lem:shell}, there is a shell $\mathcal S_E = \{E_0, E_1, \cdots, E_p\}$ centered at $E$, in which $D$ is $E_1$ (here we have $q' = q = 1$).
By Lemma \ref{lem:common_dual}, $D$ and $E$ have a common dual disk.

\begin{center}
\begin{overpic}[width=11.5cm, clip]{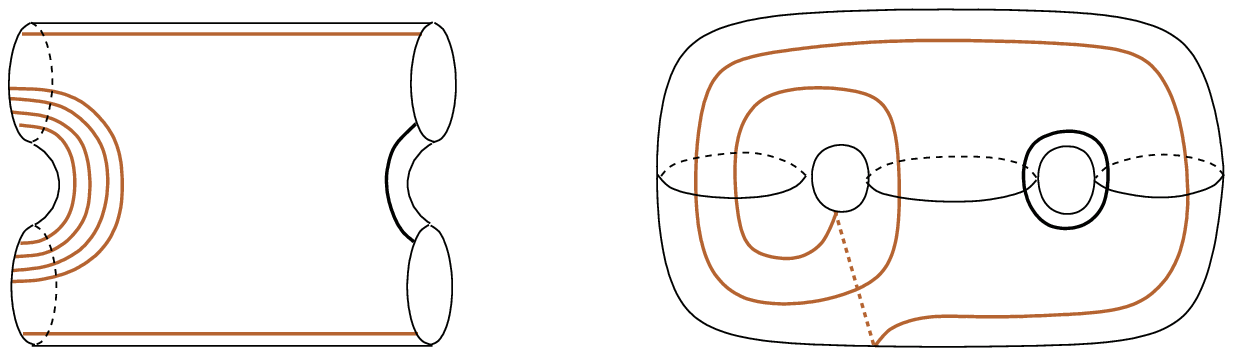}
  \linethickness{3pt}
  \put(65,0){(a)}
  \put(240,0){(b)}
  \put(47,52){$\partial D$}
  \put(90,80){$\partial D$}
  \put(90,22){$\partial D$}
  \put(67,25){$\Sigma'$}
  \put(90,52){$\partial E$}
  \put(-5,78){$\partial E_0'$}
  \put(-5,27){$\partial E_0'$}
  \put(128,78){$\partial E'$}
  \put(128,27){$\partial E'$}
  \put(156,52){$E_0'$}
  \put(316,52){$E'$}
  \put(241,40){$E''$}
  \put(265,34){$\partial E$}
  \put(240,80){$\partial D$}
  \put(310,20){$W$}
\end{overpic}
\captionof{figure}{(a) $\partial E$ and $\partial D$ lying in the $4$-holed sphere $\Sigma'$ (when $p = 5$ for example). (b) Two common dual disks $E'$ and $E''$ of $D$ and $E$ for $L(2, 1)$.}
\label{common_dual}
\end{center}

Now, let $E'$ be a common dual disk of $D$ and $E$.
Let $E'_0$ be the unique semiprimitive disk in $W$ disjoint from $E \cup E'$.
We recall that $E_0'$ is the meridian disk of the solid torus $\cl(W - \Nbd(E'))$.
Then $\partial E'_0$ intersects $\partial D$ in $p$ points.
Cut the surface $\partial W$ along the boundary circles $\partial E'$ and $\partial E'_0$ to obtain the $4$-holed sphere $\Sigma'$.
In $\Sigma'$, the boundary circle $\partial E$ is a single arc connecting two boundary circles of $\Sigma'$ that came from $\partial E'$.
But the boundary circle $\partial D$ in $\Sigma'$ consists of $(p-1)$ arcs connecting two boundary circles that came from $\partial E'_0$ together with two arcs connecting $\partial E'$ and $\partial E'_0$ as in Figure \ref{common_dual} (a).
Observe that if there is a common dual disk of $D$ and $E$ other than $E'$, then it cannot intersect $E' \cup E'_0$ otherwise it intersects $\partial D$ or $\partial E$ in more than one points.
Thus the boundary of any common dual disk $E''$ of $D$ and $E$ other than $E'$ is a circle inside $\Sigma'$, and hence, from the figure, it is obvious that one more common dual disk $E''$ other than $E'$ exists if and only if $p = 2$, and such an $E''$ is unique in this case.
See Figure \ref{common_dual} (b).

Conversely, suppose that every primitive pair has a common dual disk.
Choose any shell $\mathcal S_E = \{E_0, E_1, \cdots, E_p\}$ in $\mathcal D(V)$ centered at a primitive disk $E$.
Then one of the disks $E_{q'}$ and $E_q$ is primitive, where $q'$ satisfies $1 \leq q' \leq p/2$ and $qq' = \pm 1 \pmod{p}$, which forms a primitive pair with $E$.
If $\{E, E_{q'}\}$ is a primitive pair, then it has a common dual disk, and so, by Lemma \ref{lem:common_dual}, there is a semiprimitive disk in $V$ disjoint from $E$ and $E_{q'}$.
The only possible semiprimitive disk disjoint from $E$ and $E_{q'}$ is $E_{q'-1}$ or $E_{q'+1}$ by Lemma \ref{lem:surgery_on_primitive}, that is, $E_{q' - 1} = E_0$ or $E_{q'+1} = E_p$.
In any cases, we have $q = 1$ (the latter case implies $(p, q) = (2, 1)$ since we assumed $1 \leq q' \leq p/2$).
The same conclusion holds in the case where $\{E, E_q\}$ is a primitive pair.
\end{proof}

It is clear that any primitive disk is a member of infinitely many primitive pairs.
But a primitive pair can be contained at most two primitive triples, which is shown as follows.

\begin{theorem}
Given a lens space $L(p, q)$, for $1 \leq q \leq p/2$, with a genus-$2$ Heegaard splitting $(V, W; \Sigma)$ of $L(p, q)$, there is a primitive triple in $V$ if and only if $q = 2$ or $p= 2q +1$.
In this case, we have the following refinements.
\begin{enumerate}
\item If $p = 3$, then each primitive pair is contained in a unique primitive triple.
\item If $p = 5$, then each primitive pair having a common dual disk is contained in a unique primitive triple, and each having no common dual disk is contained in exactly two primitive triples.
\item If $p \geq 7$, then each primitive pair having a common dual disk is contained either in a unique or in no primitive triple, and each having no common dual disk is contained in a unique primitive triple.
\item Further, if $p = 3$, then each of the three primitive pairs in any primitive triple in $V$ has a unique common dual disk, which form a primitive triple in $W$.
    If $p \geq 5$, then exactly one of the three primitive pairs in any primitive triple has a common dual disk, which is unique.
\end{enumerate}
\label{thm:triple}
\end{theorem}

\begin{proof}
Note that $L(2q + 1, q)$ is homeomorphic to $L(2q+1, 2)$.
We prove first the ``if'' part together with the refinements.
Suppose that $q=2$ or $p=2q+1$, and let $\{D, E\}$ be any primitive pair of $V$.
By Lemma \ref{lem:shell}, there is a unique shell $\mathcal S_E = \{E_0, E_1, \cdots, E_p\}$ centered at $E$ containing $D$.  We may assume that $D$ is one of $E_1$, $E_2$ or $E_q$.

\smallskip

\noindent (1) If $p=3$, the disk $D$ is $E_1$, and so $E_2$ is the unique primitive disk disjoint from $E \cup E_1$ by Lemma \ref{lem:surgery_on_primitive}.
Thus $\{D, E\}$ is contained in the unique primitive triple $\{D, E, E_2\}$.

\smallskip

\noindent (2) If $p=5$, then the disk $D$ is either $E_1$ or $E_2$.
If $\{D, E\}$ has a common dual disk, then $D$ is $E_1$, and they are contained in the unique primitive triple $\{D, E, E_2\}$.
If $\{D, E\}$ has no common dual disk, then $D$ is $E_2$, and they are contained in exactly two primitive triples $\{D, E, E_1\}$ and $\{D, E, E_3\}$.

\smallskip

\noindent (3) If $p \geq 7$, then $D$ is either $E_1$, $E_2$ or $E_q$.
Observe that if one of $E_2$ and $E_q$ is primitive, then the other is not, while $E_1$ is always primitive.
If $\{D, E\}$ has no common dual disk, then $D$ is $E_2$ or $E_q$.
In this case, $\{D, E\}$ is contained in the unique primitive triple $\{D, E, E_1\}$ if $D$ is $E_2$, or in the unique triple $\{D, E, E_{q+1}\}$ if $D$ is $E_q$.
Suppose next that $\{D, E\}$ has a common dual disk.
Then $D$ is $E_1$, and hence $\{D, E\}$ is either contained in a unique primitive triple or contained in no primitive triple, according as $E_2$ is primitive or not.

\smallskip

\noindent (4) Let $\{D, E, F\}$ be any primitive triple in $V$, and let $\mathcal S_E = \{E_0, E_1, \cdots, E_p\}$ be the unique shell centered at $E$ containing $D$.  Again, we may assume that $D$ is one of $E_1$, $E_2$ or $E_q$.
Suppose that $p=3$.
Then we have $D = E_1$ and $F = E_2$ in the shell $\mathcal S_E = \{E_0, E_1, E_2, E_3\}$.
The primitive pairs $\{E, D\} = \{E, E_1\}$ and $\{E, F\} = \{E, E_2\}$ in the triple have unique common dual disks, say $E'$ and $E''$ respectively, by Lemma \ref{lem:common_dual} and Theorem \ref{thm:common_dual}.
Further, $\{E', E''\}$ is a primitive pair in $W$ (in fact, $\partial E''$ is the circle $e''$ in the proof of Lemma \ref{lem:sequence}).
Furthermore, exchanging the roles of $D$ and $E$, there exists the unique shell $\mathcal S_D = \{D_0, D_1, D_2, D_3\}$ centered at $D$ containing $E$. Here we have $D = E_1$, $D_0 = E_0$, $D_1 = E$ and $D_2 = E_2 = F$.
The primitive pair $\{D, D_2\} = \{D, F\}$ has a unique common dual disk, say $E'''$, forms a primitive pair $\{E', E'''\}$ with the common dual disk $E'$ of $\{D, E\} = \{D, D_1\}$.
Finally, considering the unique shell centered at $F$ containing $E$, we see that $\{E'', E'''\}$ is also a primitive pair in $W$.
Thus $\{E', E'', E'''\}$ is a primitive triple in $W$.

Next, suppose that $p \geq 5$, and let $\{D, E, F\}$ be any primitive triple of $V$.
Suppose, for contradiction, that at least two of the primitive pairs, say $\{D, E\}$ and $\{E, F\}$, in the triple have common dual disks.
Then, in the unique shell $\mathcal S_E = \{E_0, E_1, \cdots, E_p\}$ centered at $E$ containing $D$, the disk $D$ must be $E_1$ by Lemma \ref{lem:common_dual}.
Moreover, the disk $F$ is $E_2$ by Lemma \ref{lem:surgery_on_primitive}, and the disk $E_3$ is semiprimitive, that is, $E_p$ by Lemma \ref{lem:common_dual} again. Thus, we must have $p = 3$, a contradiction.

\smallskip

Conversely, suppose that there is a primitive triple $\{D, E, F\}$ in $V$.
Again, we consider the unique shell  $\mathcal S_E = \{E_0, E_1, \cdots, E_p\}$ centered at $E$ containing $D$.
Then $\mathcal S_E$ is a $(p, \bar{q})$-shell for some $\bar q \in \{q, q', p-q', p-q\}$, where $q'$ is the unique integer satisfying $qq' \equiv \pm 1$ (mod $p$) and $1 \leq q' \leq p/2$.
We first consider the case $\bar{q} = q$.
Then we may assume that $D$ is $E_1$ or $E_{q'}$ by Lemma \ref{lem:shell}.
If $D$ is $E_1$, then $F$ is $E_2$ by Lemma \ref{lem:surgery_on_primitive}, and so $q' = 2$ by Lemma \ref{lem:sequence}. Thus $p = 2q + 1$.
If $D$ is $E_{q'}$, then $F$ is $E_{q'-1}$ or $E_{q' + 1}$ by Lemma \ref{lem:surgery_on_primitive} again. That is, $q' - 1 = 1$ or $q' + 1 = p-q'$ by Lemma \ref{lem:sequence} again. Thus $p = 2q + 1$ or $q = 2$. We have the same argument for the other cases,  $\bar q \in \{ q', p-q', p-q\}$.
\end{proof}

Now we are ready to give a precise description of the primitive disk complex $\mathcal P(V)$ for the genus-$2$ Heegaard splitting of each lens space.
For convenience, we classify all the edges and $2$-simplices of $\mathcal P(V)$ as follows.

\begin{enumerate}
\item An edge of $\mathcal P(V)$ is called an {\it edge of type-$0$} ({\it type-$1$, type-$2$,} respectively) if a primitive pair representing the end vertices of the edge has no common dual disk (has a unique common dual disk, has exactly two common dual disks which form a primitive pair in $W$, respectively).
\item A $2$-simplex of $\mathcal P(V)$ is called a {\it $2$-simplex of type-$1$} ({\it of type-$3$,} respectively) if exactly one of the three primitive pairs in the primitive triple representing the three edges of the $2$-simplex has a unique common dual disk (if all the three pairs have unique common dual disks which form a primitive triple in $W$, respectively).
\end{enumerate}

By Theorems \ref{thm:common_dual} and \ref{thm:triple}, we see that each of the edges and $2$-simplices of $\mathcal P(V)$ is one of those types in the above.
In the following theorem, we describe the combinatorial structure of $\mathcal P(V)$ for each of the lens spaces, which is a direct consequence of Theorems \ref{thm:contractible}, \ref{thm:common_dual} and \ref{thm:triple}.

\begin{theorem}
Given any lens space $L(p, q)$, $1 \leq q \leq p/2$, with a genus-$2$ Heegaard splitting $(V, W; \Sigma)$, if $p \equiv \pm 1 \pmod q$, then the primitive disk complex $\mathcal P(V)$ is contractible and we have one of the following cases.
\begin{enumerate}
\item If $q \neq 2$ and $p \neq 2q + 1$, then $\mathcal P(V)$ is a tree, and every vertex has infinite valency.
    In this case,
    \begin{enumerate}
    \item if $p=2$ and $q=1$, then every edge is of type-$2$.
    \item if $p \geq 4$ and $q=1$, then every edge is of type-$1$.
    \item if $q \neq 1$, then every edge is of either type-$0$ or type-$1$, and infinitely many edges of type-$0$ and of type-$1$ meet in each vertex.
    \end{enumerate}
\item If $q = 2$ or $p=2q+1$, then $\mathcal P(V)$ is $2$-dimensional, and every vertex meets infinitely many $2$-simplices.
    In this case,
    \begin{enumerate}
    \item if $p = 3$, then every edge is of type-$1$, every $2$-simplex is of type-$3$, and every edge is contained in a unique $2$-simplex.
    \item if $p = 5$, then every edge is of either type-$0$ or type-$1$, and every $2$-simplex is of type-$1$. Every edge of type-$0$ is contained in exactly two $2$-simplices, while every edge of type-$1$ in a unique $2$-simplex.
    \item if $p \geq 7$, then every edge is of either type-$0$ or type-$1$, and every $2$-simplex is of type-$1$.
        Every edge of type-$0$ is contained in a unique $2$-simplex.
        Every edge of type-$1$ is contained in a unique $2$-simplex or in no $2$-simplex.
    \end{enumerate}
\end{enumerate}
If $p \not\equiv \pm 1 \pmod q$, then $\mathcal P(V)$ is not connected, and it consists of infinitely many tree components.
All the tree components are isomorphic to each other. Any vertex of $\mathcal P(V)$ has infinite valency, and further, infinitely many edges of type-$0$ and of type-$1$ meet in each vertex.
\label{thm:structure}
\end{theorem}

Figure \ref{shape} illustrates a portion of each of the contractible primitive disk complexes $\mathcal P(V)$ classified in the above, together with its surroundings in $\mathcal D(V)$.
We label simply $E$ or $E_j$ for the vertices represented by disks $E$ or $E_j$.
In the case (2)-(b), the complex $\mathcal P(V)$ for $L(5, 2)$, every edge is contained a unique ``band''.
The edges in the boundary of a band are of type-$1$, while the edges inside a band are of type-$0$.
The whole figure of $\mathcal P(V)$ for $L(5, 2)$ can be imagined as the union of infinitely many bands such that any of two bands are disjoint from each other or intersects in a single vertex.
In the case (2)-(c), there are two kind of shells $\mathcal S_E = \{E_0, E_1, \cdots, E_p\}$ in $\mathcal P(V)$ centered at a primitive disk $E$.
The first one has primitive disks $E_1, E_q, E_{p-q}$ and $E_{p-1}$, while the second one has $E_1, E_2, E_{p-2}$ and $E_{p-1}$.
Figure \ref{shape} (2)-(c) illustrates an example of the first one.

\begin{center}
\begin{overpic}[width=10cm, clip]{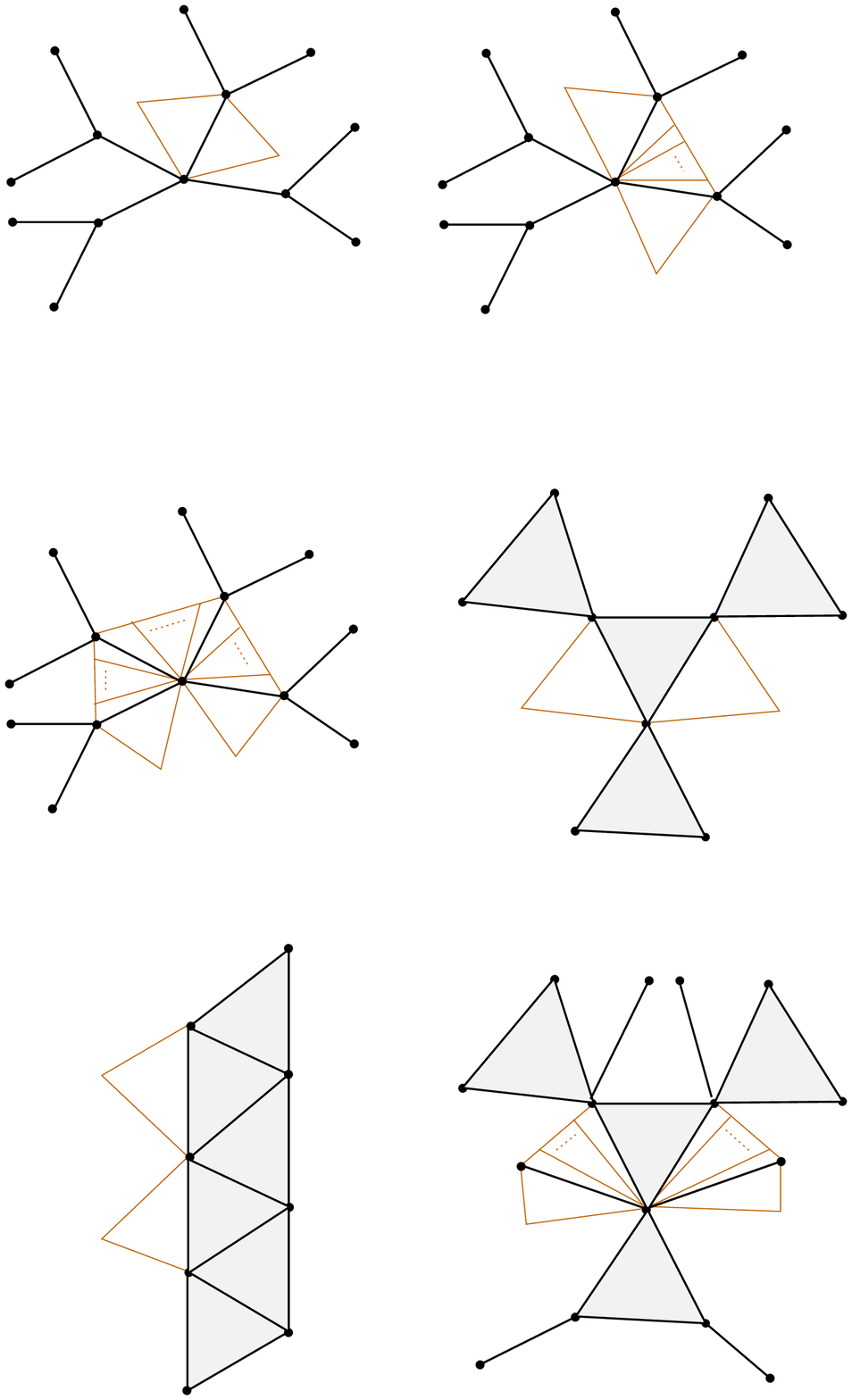}
  \linethickness{3pt}

     \put(46,373){(1)-(a)}
     \put(57,428){\small $E$}
     \put(32,463){\small $E_0$}
     \put(72,473){\small $E_1$}
     \put(93,446){\small $E_2$}
     \put(45,423){\small $2$}
     \put(76,426){\small $2$}
     \put(103,416){\small $2$}
     \put(108,438){\small $2$}
     \put(25,403){\small $2$}
     \put(14,415){\small $2$}
     \put(15,437){\small $2$}
     \put(19,461){\small $2$}
     \put(40,439){\small $2$}
     \put(61,476){\small $2$}
     \put(90,466){\small $2$}
     \put(62,451){\small $2$}

     \put(200,373){(1)-(b)}
     \put(199,426){\small $E$}
     \put(176,468){\small $E_0$}
     \put(217,472){\small $E_1$}
     \put(233,419){\small $E_{p-1}$}
     \put(215,397){\small $E_p$}
     \put(191,422){\small $1$}
     \put(219,426){\small $1$}
     \put(252,425){\small $1$}
     \put(254,439){\small $1$}
     \put(171,403){\small $1$}
     \put(160,415){\small $1$}
     \put(162,437){\small $1$}
     \put(165,461){\small $1$}
     \put(186,438){\small $1$}
     \put(207,476){\small $1$}
     \put(236,466){\small $1$}
     \put(208,451){\small $1$}

     \put(46,198){(1)-(c)}
     \put(57,255){\small $E$}
     \put(48,229){\small $E_0$}
     \put(18,258){\small $E_1$}
     \put(98,262){\small $E_{p-1}$}
     \put(75,233){\small $E_p$}
     \put(45,254){\small $1$}
     \put(76,259){\small $1$}
     \put(103,248){\small $1$}
     \put(102,277){\small $0$}
     \put(25,233){\small $0$}
     \put(14,247){\small $1$}
     \put(15,278){\small $0$}
     \put(19,292){\small $1$}
     \put(40,270){\small $0$}
     \put(61,306){\small $0$}
     \put(90,297){\small $1$}
     \put(69,278){\small $0$}

     \put(200,198){(2)-(a)}
     \put(219,252){\small $E$}
     \put(163,255){\small $E_0$}
     \put(198,295){\small $E_1$}
     \put(227,295){\small $E_2$}
     \put(261,254){\small $E_3$}
     \put(213,220){\small $1$}
     \put(198,234){\small $1$}
     \put(229,233){\small $1$}
     \put(202,267){\small $1$}
     \put(230,267){\small $1$}
     \put(215,293){\small $1$}
     \put(174,284){\small $1$}
     \put(164,316){\small $1$}
     \put(194,308){\small $1$}
     \put(270,311){\small $1$}
     \put(260,282){\small $1$}
     \put(242,308){\small $1$}

     \put(46,7){(2)-(b)}
     \put(51,105){\small $E$}
     \put(21,132){\small $E_0$}
     \put(49,154){\small $E_1$}
     \put(99,132){\small $E_2$}
     \put(100,87){\small $E_3$}
     \put(49,63){\small $E_4$}
     \put(21,77){\small $E_5$}
     \put(76,44){\small $0$}
     \put(76,65){\small $0$}
     \put(76,85){\small $0$}
     \put(76,105){\small $0$}
     \put(76,127){\small $0$}
     \put(76,149){\small $0$}
     \put(76,168){\small $0$}
     \put(55,46){\small $1$}
     \put(55,85){\small $1$}
     \put(55,128){\small $1$}
     \put(98,68){\small $1$}
     \put(98,111){\small $1$}
     \put(98,153){\small $1$}

     \put(200,7){(2)-(c)}
     \put(211,79){\small $E$}
     \put(165,81){\small $E_0$}
     \put(161,103){\small $E_1$}
     \put(185,131){\small $E_q$}
     \put(243,131){\small $E_{q+1}$}
     \put(263,105){\small $E_{p-1}$}
     \put(261,87){\small $E_p$}
     \put(213,57){\small $0$}
     \put(198,71){\small $1$}
     \put(229,70){\small $0$}
     \put(202,104){\small $0$}
     \put(230,104){\small $0$}
     \put(215,130){\small $1$}
     \put(174,121){\small $1$}
     \put(164,153){\small $0$}
     \put(194,145){\small $0$}
     \put(270,148){\small $0$}
     \put(260,119){\small $0$}
     \put(242,145){\small $1$}
     \put(173,49){\small $1$}
     \put(190,93){\small $1$}
     \put(245,95){\small $1$}
     \put(246,45){\small $1$}
     \put(208,143){\small $1$}
     \put(227,143){\small $1$}
\end{overpic}


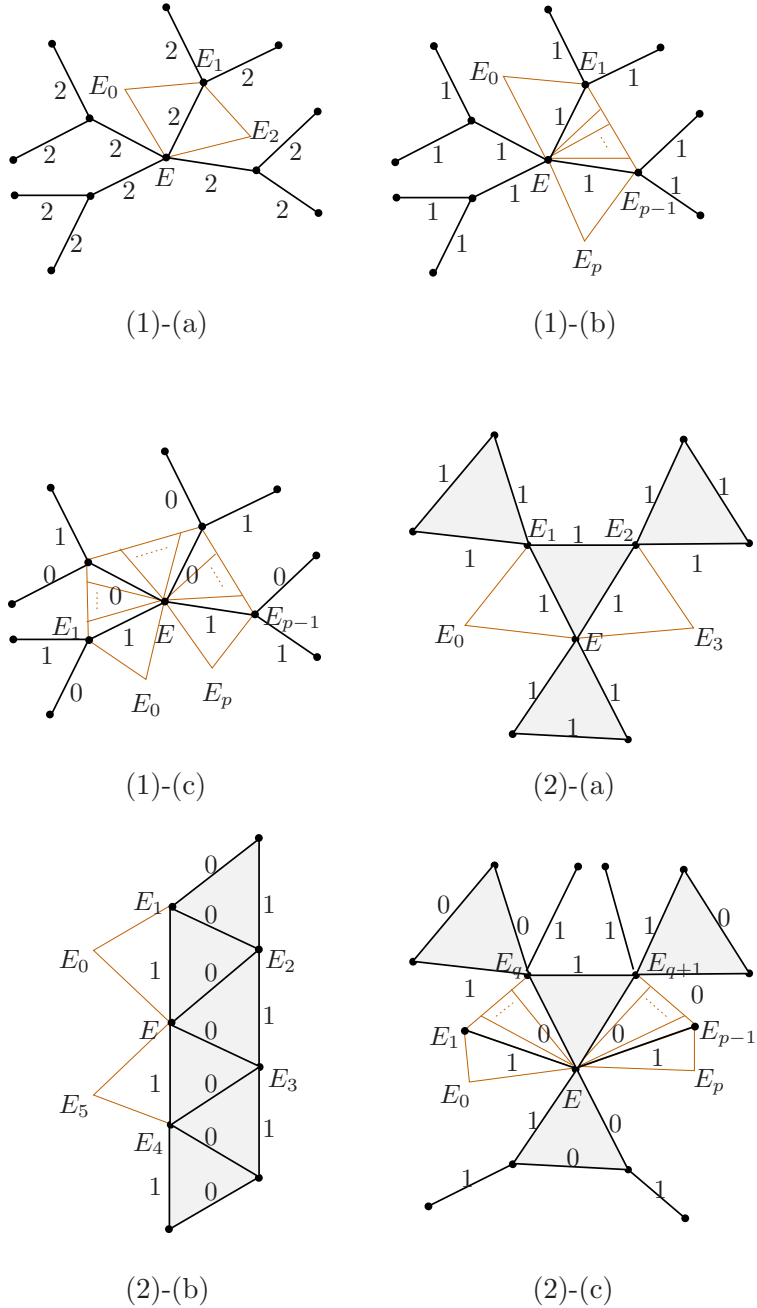
\captionof{figure}{A portion of each primitive disk complex $\mathcal P(V)$ together with the associated shells in $\mathcal D(V)$. Each number designates the type of the edge.}
\label{shape}
\end{center}



\section{The genus-2 Goeritz groups of lens spaces}
\label{sec:Goeritz_groups}

\subsection{The primitive disks under the action of the Goeritz group}
\label{subsec:the_primitive_disks_under_the_action_of_the_Goeritz_group}

By Bonahon-Otal \cite{BO83} each lens space admits
a unique Heegaard surface of each genus $g \geqslant 1$ up to isotopy.
Further, they showed that the two handlebodies of each genus-$g$ Heegaard splitting
are isotopic to each other when $g \geq 2$.
However, the genus-1 Heegaard splitting of a lens space is somewhat more rigid in the following sense.
\begin{lemma}[Bonahon \cite{Bon83}]
\label{lem:Uniqueness of the Heegaard splitting of a lens space}
There exists an orientation-preserving
homeomorphism $\iota$ of $L(p,q)$ that exchanges the two solid tori of
the genus-$1$ Heegaard splitting if and only if
$q^2 \equiv 1 \pmod p$.
\end{lemma}

Given a genus-$g$ Heegaard splitting of a 3-manifold, the {\it Goeritz group} of the splitting is the group of isotopy classes
of orientation-preserving homeomorphisms of the manifold that preserve each of the handlebodies of the splitting setwise.
By Lemma \ref{lem:Uniqueness of the Heegaard splitting of a lens space},
the Goeritz group of a splitting for each lens space depends only on the genus of the splitting, and hence we say {\it the genus-$g$ Goeriz group of a lens space} without mentioning a specific genus-$g$ splitting of it.
We denote by $\mathcal{G} = \mathcal{G}_{L(p,q)}$
the genus-2 Goeritz group of $L(p,q)$.
We recall that $(V, W; \Sigma)$ is a genus-$2$ Heegaard splitting of a lens space $L(p, q)$ with $1 \leq q \leq p/2$.
We denote by $V_D$ the solid torus $\cl(V - \Nbd(D))$ where $D$ is an essential nonseparating disk in $V$.

Throughout the section, we will assume that $p \equiv \pm 1 \pmod q$, that is, the primitive disk complex $\mathcal P(V)$ is connected. Further we fix the following:
\begin{itemize}
\item
A primitive disk $E$ in $V$.
\item
A $(p, q)$-shell $\mathcal{S}_E = \{E_0, E_1, \ldots, E_p \}$ centered at $E$.
\item
The unique $(p, q')$-shell $\mathcal{S}_D = \{ D_0, D_1, \ldots, D_p \}$ centered at $D = E_{q'}$
such that $E=D_q$, where $q'$ is the unique integer satisfying $qq' \equiv \pm 1 \pmod p$ and $1 \leq q' \leq p/2$.
\end{itemize}
We use the above four primitive disks $E$, $D$, $E_1$, $D_1$ to describe the orbits of the action of
the genus-$2$ Goeritz group to the set of primitive pairs.
Note that if $q=1$, then $D = E_1$ and $E=D_1$.

\begin{lemma}
\label{thm:number of orbits of primitive disks}
If $q^2 \equiv 1 \pmod p$, the action of the Goeritz group $\mathcal{G}$ on
the set of vertices of the primitive disk complex $\mathcal P(V)$ is transitive.
If $q^2 \not\equiv 1 \pmod p$, the action of $\mathcal{G}$ on
the set of vertices of $\mathcal P(V)$ has exactly two orbits
$\mathcal{G} \cdot \{E\}$ and $\mathcal{G} \cdot \{D\}$.
\end{lemma}
\begin{proof}
Suppose first that $q^2 \equiv 1 \pmod p$.
By Lemma \ref{lem:Uniqueness of the Heegaard splitting of a lens space},
there exists an orientation-preserving homeomorphism $\iota$ of $L(p,q)$
that exchanges the solid tori of a genus-$1$ Heegaard splitting.
By the uniqueness of the genus-$2$ Heegaard splitting for $L(p,q)$ up to isotopy,
we can assume that $\iota$ preserves $V$, i.e. $\iota \in \mathcal{G}$.
Let $F$ be an arbitrary primitive disk in $V$.
Then the solid torus $V_F$ (and $V_{\iota (F)}$)
is isotopic to $V_E$.
Thus by the uniqueness of stabilization,
there exists an element $f \in \mathcal{G}$
such that $f(E) = F$ or $\iota (F)$.
This implies that $\{F\} \in \mathcal{G} \cdot \{E\}$.

Next, suppose that $q^2 \not\equiv 1 \pmod p$.
As in the proof of Lemma \ref{lem:shells_crossing},
$V_D$ is isotopic to the exterior of $V_E$ in $L(p,q)$.
If there exists an element $f \in \mathcal{G}$
such that $f(D) = E$, then $f$ maps $V_D$ to $V_E$, which contradicts
Lemma \ref{lem:Uniqueness of the Heegaard splitting of a lens space}.
\end{proof}

%
%

\begin{lemma}
\begin{enumerate}
\item
If $q = 1$,
the action of the Goeritz group $\mathcal{G}$ on the set of edges of the
primitive disk complex $\mathcal P(V)$ is transitive.
The two end points of the edge $\{ E, D \}$ can be exchanged
by the action of $\mathcal{G}$.
\item
If $q \neq 1$ and $q^2 \equiv 1 \pmod p$,
the action of $\mathcal{G}$
on the set of edges of $\mathcal P(V)$ has
exactly $2$ orbits $\mathcal{G} \cdot \{ E, D \}$ and $\mathcal{G} \cdot \{ E, E_1 \}$.
The two end points of each of the edges $\{ E, D \}$ and $\{ E, E_1 \}$ can be exchanged
by the action of $\mathcal{G}$.
\item
Otherwise,
the action of $\mathcal{G}$ on the set of edges of $\mathcal P(V)$ has
exactly $3$ orbits $\mathcal{G} \cdot \{ E, D \}$,
$\mathcal{G} \cdot \{ E, E_1 \}$ and $\mathcal{G} \cdot \{ D, D_1 \}$.
The two end points of each of the edges $\{ E, E_1 \}$ and $\{ D, D_1 \}$
can be exchanged by the action of $\mathcal{G}$ whereas those of $\{ E, D \}$ cannot.
\end{enumerate}
\label{lem:number of orbits of primitive pairs}
\end{lemma}
\begin{proof}
(1) Let $\{ A , B \}$ be a primitive pair.
Then by Lemma \ref{lem:shell}, there exists a unique shell $\mathcal S_B = \{B_0,  B_1, B_2, \ldots, B_p \}$
centered at $B$ containing $A$.
Without loss of generality, we may assume that $A = B_1$.
By the definition of shells,
we have $\{A, B \} \in \mathcal{G} \cdot \{ E, E_1 \}$.
Since in this case we have $q = q' = 1$,
it follows from Lemma \ref{lem:shells_crossing} that
the two end points of the edge $\{ E, E_1 \}$ can be exchanbed
by the action of $\mathcal{G}$.

\noindent (2) In this case, we have $q = q' \neq 1$.
Let $\{ A , B \}$ be a primitive pair.
Then by Lemma \ref{lem:shell}, there exists a unique shell $\mathcal S_B = \{B_0,  B_1, B_2, \ldots, B_p \}$
centered at $B$ containing $A$.
Without loss of generality, we may assume that $A = B_1$ or $B_q$.
It follows directly from the definition of shells
that in the former case we have $\{A, B \} \in \mathcal{G} \cdot \{ E, E_1 \}$, and
in the latter case we have $\{A, B \} \in \mathcal{G} \cdot \{ E, D \}$.
Since the primitive pair $\{E, E_1\}$ admits a common dual disk whereas the pair $\{E, D\}$ does not,
we see that  $\mathcal{G} \cdot \{ E, D \} \cap \mathcal{G} \cdot \{ E, E_1 \} = \emptyset$.
By Lemma \ref{lem:shells_crossing},
the two end points of each of the edges $\{ E, D \}$ and $\{ E, E_1 \}$ can be exchanged
by the action of $\mathcal{G}$.

\noindent (3)
In this case we have $q \neq q'$, $q > 1$ and $q' > 1$.
Let $\{ A , B \}$ be a primitive pair.
Then by Lemma \ref{lem:shell}, there exists a unique shell $\mathcal S_B = \{B_0,  B_1, B_2, \ldots, B_p \}$
centered at $B$ containing $A$.
Without loss of generality, we may assume that $A = B_i$, where $1 \leq i \leq p/2$.
Again by the definition of shells we have:
\begin{description}
\item[{\it Case }$1$]
If $\mathcal S_B$ is a $(p,q)$-shell and $A=B_1$, then $\{A, B \} \in \mathcal{G} \cdot \{ E, E_1 \}$.
\item[{\it Case }$2$]
If $\mathcal S_B$ is a $(p,q')$-shell and $A=B_1$, then $\{A, B \} \in \mathcal{G} \cdot \{ D, D_1 \}$.
\item[{\it Case }$3$]
If $\mathcal S_B$ is a $(p,q)$-shell and $A=B_{q'}$, or
if $\mathcal S_B$ is a $(p,q')$-shell and $A=B_{q}$,
then $\{A, B \} \in \mathcal{G} \cdot \{ E, D \}$.
\end{description}
By Lemma \ref{lem:shells_crossing},
the two end points of each of the edges $\{ E, E_1 \}$ and $\{ D, D_1 \}$ can be exchanged
by an involution of $\mathcal{G}$.
Since $\mathcal{G} \cdot \{E\} \cap \mathcal{G} \cdot \{D\} = \emptyset$ by
Lemma \ref{thm:number of orbits of primitive disks},
the two end points of $\{ E, D \}$ cannot be exchanged.
\end{proof}

\subsection{Presentations of the Goertiz gorups}
\label{subsec:presentations}

The following is a specialized version of Bass-Serre Structure Theorem, which is the key to obtain a  presentation
of the Goeritz group $\mathcal{G}$.
\begin{theorem}[Serre \cite{S}]
\label{thm:theorem by Brown}
Suppose that a group $G$ acts on a tree $\mathcal{T}$
without inversion on the edges.
If there exists a subtree $\mathcal{L}$ of $\mathcal{T}$
such that every vertex $($every edge, respectively$)$ of $\mathcal{T}$ is equivalent modulo $G$
to a unique vertex $($a unique edge, respectively$)$ of $\mathcal{L}$.
Then $G$ is the free product of the isotropy groups $G_v$ of the
vertices $v$ of $\mathcal{L}$, amalgamated along the isotropy groups $G_e$
of the edges $e$ of $\mathcal{L}$.
\end{theorem}

In the following we will denote by $\mathcal{G}_{\{A_1 , A_2 , \ldots , A_k \}}$ the subgroup of
the genus-2 Goeritz group $\mathcal{G}$ consisting of elements that preserve each of
$A_1$, $A_2 , \ldots , A_k$ setwise, where each $A_i$ will be a disk or the union of disks in $V$ or $W$.

\begin{lemma}
\label{lem:stabilizer of a primitive disk}
Let $A$ be a primitive disk in $V$.
Then we have
$\mathcal{G}_{\{ A \}} =
\langle \alpha \mid \alpha^2 \rangle
\oplus \langle \beta, \gamma \mid
{\gamma}^2 \rangle$, where
$\alpha$ is the hyperelliptic involution of both
$V$ and $W$, $\beta$ is the half-twist along a reducing sphere, and
$\gamma$ exchanges two disjoint dual disks of $A$ as described in Figure ${\rm \ref{fig:isotropy_of_a_primitive_disk}}$.
\end{lemma}
\begin{center}
\begin{overpic}[width=14cm, clip]{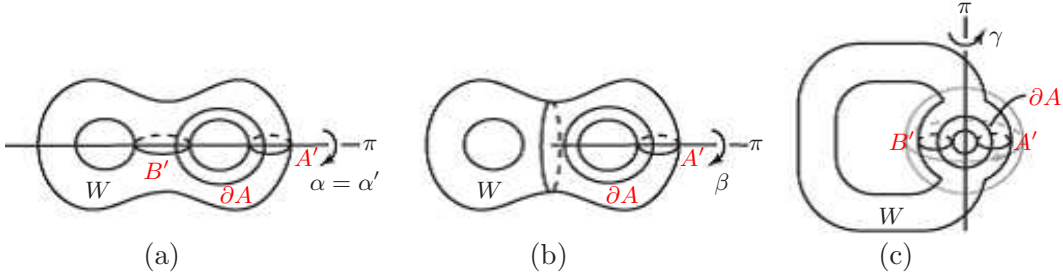}
  \linethickness{3pt}
  \put(52,0){(a)}
  \put(30,25){\footnotesize $W$}
  \put(52,32){\footnotesize {\color{red}$B'$}}
  \put(81,23){\footnotesize {\color{red}$\partial A$}}
  \put(109,37){\footnotesize {\color{red}$A'$}}
  \put(135,43){\footnotesize $\pi$}
  \put(115,28){\footnotesize $\alpha = \alpha'$}

  \put(198,0){(b)}
  \put(178,25){\footnotesize $W$}
  \put(227,23){\footnotesize {\color{red}$\partial A$}}
  \put(255,37){\footnotesize {\color{red}$A'$}}
  \put(281,43){\footnotesize $\pi$}
  \put(268,28){\footnotesize $\beta$}

  \put(330,0){(c)}
  \put(330,15){\footnotesize $W$}
  \put(334,43){\footnotesize {\color{red}$B'$}}
  \put(381,43){\footnotesize {\color{red}$A'$}}
  \put(387,60){\footnotesize {\color{red}$\partial A$}}
  \put(360,95){\footnotesize $\pi$}
  \put(372,84){\footnotesize $\gamma$}
\end{overpic}
\captionof{figure}{Generators of $\mathcal{G}_{\{ A \}}$.}
\label{fig:isotropy_of_a_primitive_disk}
\end{center}
\begin{proof}
Since the argument is almost the same as Lemma 5.1 of \cite{C2},
we explain the outline.
Let $\mathcal{P}_{A}$ be the full-subcomplex of $\mathcal{D} (W)$
spanned by all the dual disks of $A$.
Then we can show that any dual disk of $A$ in $W$ is disjoint from the
unique semiprimitive disk $A_0'$ disjoint from $\partial A$, which implies
that $\mathcal{P}_A$ is 1-dimensional.
Further, $\mathcal P_A$ is a subcomplex of the disk complex for $W$ satisfying the condition in Theorem \ref{thm:surgery}, and hence $\mathcal{P}_{A}$ is a tree.
Let $\mathcal{P}_{A}^\prime$ be a first barycentric subdivision of $\mathcal{P}_{A}$.
Let $A^\prime$ and $B^\prime$ be disjoint dual disks of $A$.
The quotient of $\mathcal{P}_{A}^\prime$ by the action of
$\mathcal{G}$ is a single edge.
It follows from Theorem \ref{thm:theorem by Brown} that $\mathcal{G}_{\{A\}} =
\mathcal{G}_{\{A , A^\prime\}} *_{\mathcal{G}_{\{A , A^\prime, B^\prime\}}}
\mathcal{G}_{\{ A , A^\prime \cup B^\prime \}}$.
An easy computation shows the following:
\begin{itemize}
\item
$\mathcal{G}_{\{A , A^\prime\}} =
\langle \alpha \mid \alpha^2 \rangle
\oplus
\langle \beta \mid - \rangle $, where
$\alpha$ is the hyperelliptic involution of both $V$ and $W$, and
$\beta$ is the half-twist along the reducing sphere
$\partial ( \Nbd ( A \cup A^\prime ) )$;
see Figure \ref{fig:isotropy_of_a_primitive_disk}
(a) and (b),
\item
$\mathcal{G}_{\{ A , A^\prime \cup B^\prime \}} =
\langle \alpha^\prime \mid {\alpha^\prime}^2 \rangle
 \oplus \
\langle \gamma \mid \gamma^2 \rangle
$, where
$\alpha^\prime$ is the hyperelliptic involution of both $V$ and $W$, and
$\gamma$ exchanges $A^\prime$ and $B^\prime$;
see Figure \ref{fig:isotropy_of_a_primitive_disk}
(a) and (c),
\item
$\mathcal{G}_{\{A , A^\prime, B^\prime \}} =
\langle \alpha \mid \alpha^2 \rangle $, where
$\alpha$ is the hyperelliptic involution of both $V$ and $W$;
see Figure \ref{fig:isotropy_of_a_primitive_disk} (a).
\end{itemize}
Since the unique non-trivial element $\alpha$
of $\mathcal{G}_{\{A , A^\prime, B^\prime \}}$ provides
a relation $\alpha = \alpha^\prime$ in
the free product
$\mathcal{G}_{\{A , A^\prime\}} *
\mathcal{G}_{\{ A , A^\prime \cup B^\prime \}}$,
we obtain the required presentation of
$\mathcal{G}_{\{A\}}$.
\end{proof}

\begin{lemma}
\label{lem:stabilizers of primitive pairs}
Suppose that $p \geq 3$.
Let $\{A, B\}$ be an edge of the primitive disk complex $\mathcal P(V)$.
Then we have $\mathcal{G}_{ \{ A, B \} } = \langle \alpha \mid \alpha^2 \rangle$.
If the two end points of the edge $\{ A, B \}$
can be exchanged by the action of $\mathcal{G}$,
then we have
$\mathcal{G}_{ \{ A \cup B \} } =
\langle \alpha \mid \alpha^2 \rangle \oplus
\langle \sigma \mid \sigma^2 \rangle$, where $\sigma$ is an element of $\mathcal G$ exchanging $A$ and $B$.
Otherwise, we have
$\mathcal{G}_{ \{ A \cup B \} } =
\langle \alpha \mid \alpha^2 \rangle$.
\end{lemma}
\begin{proof}
Let $\{ A,B \}$ be an edge of $\mathcal P(V)$.
Then by Lemma \ref{lem:shell} there exists
a unique shell $\mathcal{S}_B = \{B_0 , B_1 , \ldots , B_p \}$
centered at $B$ containing $A$ such that $A$ is one of $B_0, B_1, \cdots, B_p$.
Without loss of generality, we may assume that
$A = B_i$, where $1 \leq i < p/2$.
(We assumed $p \geq 3$.)
Let $f$ be an element of $\mathcal{G}_{ \{ A, B \} }$.
By the uniqueness of the shell, we have
$f(B_j) = B_j$ for $0 \leq j \leq p$.
Let $B'$ be the unique dual disk of $B$ disjoint from $B_0$,
and let $B_0'$ be the unique semi-primitive disk disjoint from $B$ as in
the proof of Lemma \ref{lem:shell}.
Then again by the uniqueness of the shell, we have
$f(B') = B'$ and $f(B_0') = B_0'$.
If $f$ preserves an orientation of $B$, then $f$ preserves orientations
of all of $B_j$, $B'$ and $B_0'$ since
$\{B , B_{j-1}, B_j\}$ is a triple of pairwise disjoint disks cutting $V$ into two 3-balls.
Then by Alexander's trick,
$f$ is the trivial element of $\mathcal{G}$.
If $f$ reverses an orientation of $B$, then $f$ reverses orientations
of all of $B_j$, $B'$ and $B_0'$.
Then again by Alexander's trick,
$f$ is the hyperelliptic involution $\alpha$.

If the two end points of the edge $\{ A, B \}$
cannot be exchanged by the action of $\mathcal{G}$,
it is clear that
$\mathcal{G}_{ \{ A \cup B \} } = \mathcal{G}_{ \{ A , B \} } =
\langle \alpha \mid \alpha^2 \rangle$.

Suppose that there exists an element
$\sigma \in \mathcal{G}$ that exchanges
the two end points of the edge $\{ A, B \}$.
In this case,
by Lemma \ref{lem:shell} there exists
a unique shell $\mathcal{S}_A = \{ A_0 , A_1 , \ldots , A_p \}$
centered at $A$ containing $B$ such that $B = A_i$.
Using the triple $\{B , B_{i-1}, B_i\}$, we may put {\it compatible} orientations on
$B$, $B_{j-1}$ and $B_j = A$ in a sense that
the orientations are coming from an orientation of $V$ cut off by
$B \cup B_{i-1} \cup B_i$.
We may also put an orienation on $A_{i-1}$ so that
the triple $\{A , A_{i-1}, A_i\}$ with the pre-fixed orientations on $A$ and $A_j = B$ are compatible.
Since $\sigma$ maps the shell $\mathcal{S}_B = \{B_0 , B_1 , \ldots , B_p \}$
to the shell $\mathcal{S}_A$ we see that
$\sigma \mid _B : B \to A$ is orientation-presering if and only if
so is $\sigma \mid _A : A \to B$.
This implies that $\sigma^2 = 1 \in \mathcal{G}$.
Let $\sigma_1$ and $\sigma_2$ be elements of $\mathcal{G}$ that
interchanges $D$ and $E$.
Then $\sigma_1 \sigma_2 = 1 $ or $\alpha$.
This implies $\sigma_1 = \sigma_2$ or $\alpha \sigma_1 = \sigma_2$.
Therefore we have $\mathcal{G}_{ \{ A \cup B \} } =
\langle \alpha \mid \alpha^2 \rangle \oplus
\langle \sigma \mid \sigma^2 \rangle$.
\end{proof}

We remark that, in the case of $p = 2$ or $q = 1$, the presentations of $\mathcal{G}_{ \{ A, B \} }$ and $\mathcal{G}_{ \{ A, B \} }$ have been obtained in Lemmas 5.2 and 5.3 in \cite{C2}.
Using the presentations of the isotropy groups, we have the following main theorem:

\begin{theorem}
\label{thm:presentations of the Goeritz groups}
The genus-$2$ Goeritz group $\mathcal{G}$ of a lens space $L(p, q)$, $1 \leq q \leq p/2$, with
$p \equiv \pm 1 \pmod q$
has the following presentations:
\begin{enumerate}
\item
\label{item:q1}
If $q=1$, then we have:
\begin{enumerate}
\item
\label{item:(2,1)}
$\langle \beta , \rho , \gamma \mid
{\rho}^4 , {\gamma}^2 , (\gamma \rho)^2, \rho^2 \beta \rho^2 \beta^{-1} \rangle$ if $p = 2$;
\item
\label{item:(3,1)}
$\langle \alpha \mid \alpha^2 \rangle \oplus \langle \beta , \delta , \gamma \mid
{\delta}^3 , {\gamma}^2 , (\gamma \delta)^2 \rangle$ if $p = 3$;
\item
\label{item:(4,1)}
$\langle \alpha \mid \alpha^2 \rangle \oplus \langle \beta, \gamma, \sigma \mid
{\gamma}^2, {\sigma}^2 \rangle$ if $p \geq 4$;
\end{enumerate}
\item
\label{item:q2}
If $q>1$, then we have:
\begin{enumerate}
\item
\label{item:(5,2)}
$\langle \alpha \mid \alpha^2 \rangle \oplus \langle \beta_1, \beta_2, \gamma_1, \gamma_2 \mid
{\gamma_1}^2, {\gamma_2}^2 \rangle$ if $p = 5$;
\item
\label{item:(7,2)}
$\langle \alpha \mid \alpha^2 \rangle \oplus \langle \beta_1, \beta_2, \gamma_1, \gamma_2, \sigma \mid
{\gamma_1}^2, {\gamma_2}^2, \sigma^2 \rangle$
if $p = 2q + 1$ and $q \geqslant 3$, or $p>5$ and $q=2$;
\item
\label{item:(8,3)}
$\langle \alpha \mid \alpha^2 \rangle \oplus \langle \beta, \gamma,
\sigma_1, \sigma_2 \mid
{\gamma}^2, {\sigma_1}^2, {\sigma_2}^2 \rangle$ if $q^2 \equiv 1 \pmod p$;
\item
\label{item:(10,3)}
$\langle \alpha \mid \alpha^2 \rangle \oplus \langle \beta_1, \beta_2, \gamma_1, \gamma_2,
\sigma_1, \sigma_2 \mid
{\gamma_1}^2, {\gamma_2}^2, {\sigma_1}^2, {\sigma_2}^2 \rangle$ otherwise.
\end{enumerate}
\end{enumerate}
\end{theorem}
\begin{proof}
We use the four primitive disks $E$, $D$, $E_1$ and $D_1$ defined in Section
\ref{subsec:the_primitive_disks_under_the_action_of_the_Goeritz_group}, 
but we use the same symbols $\alpha$, $\beta$, $\gamma$ and $\sigma$ in Lemmas \ref{lem:stabilizer of a primitive disk} and \ref{lem:stabilizers of primitive pairs} for the isotropy subgroups of the disks and their unions in the above.

\smallskip

\noindent (\ref{item:q1}) Since this case of $q = 1$ is already described in \cite{C2}, we briefly sketch the proof.

\smallskip

\noindent \noindent (\ref{item:q1})-(a)
By Theorem \ref{thm:structure}
the primitive disk complex $\mathcal{P}(V)$ for the genus-$2$ Heegaard splitting of $L(2,1)$ is a tree, which is described in Figure \ref{shape} (1)-(a).
Let $\mathcal{T}$ be the first barycentric subdivision
of $\mathcal{P}(V)$.
By Lemma \ref{lem:number of orbits of primitive pairs}
the quotient of ${\mathcal{T}}$ by the action of
$\mathcal{G}$ is a single edge with distinct ends.
By Theorem \ref{thm:theorem by Brown},
we have:
\[ \mathcal{G} =
\mathcal{G}_{\{ E \cup D  \}} *_{\mathcal{G}_{ \{ E, D \} }}
\mathcal{G}_{\{ E \}}.  \]
The presentation in (\ref{item:q1})-(a)
is obtained by computing each of those isotropy groups.

\smallskip

\noindent \noindent (\ref{item:q1})-(b)
By Theorem \ref{thm:structure}
the primitive disk complex $\mathcal{P}(V)$ for the genus-$2$ Heegaard splitting of $L(3,1)$ is a $2$-dimensional complex, which is described in Figure \ref{shape} (2)-(a).
In this case,
there is a deformation retraction of $\mathcal{P}(V)$ that
shrinks each 2-simplex into the cone over its 3 vertices as shown in
Figure \ref{fig:primitive_disk_complex_5}.

\begin{center}
\begin{overpic}[width=8cm, clip]{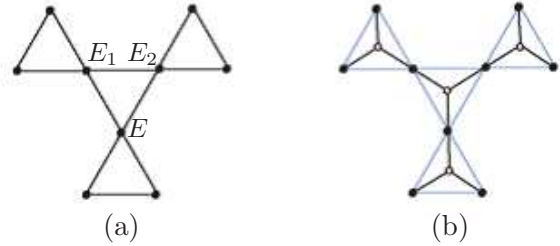}
  \linethickness{3pt}
  \put(44,0){(a)}
  \put(38,66){\small $E_1$}
  \put(53,66){\small $E_2$}
  \put(53,37){\small $E$}
  \put(168,0){(b)}
 \end{overpic}
\captionof{figure}{(a) The primitive disk complex $\mathcal{P}(V)$.
(b) The tree $\mathcal{T}$. }
\label{fig:primitive_disk_complex_5}
\end{center}
Let $\mathcal{T}$ be the resulting complex, which is a tree.
By Lemma \ref{lem:number of orbits of primitive pairs}
the quotient of ${\mathcal{T}}$ by the action of
$\mathcal{G}$ is a single edge with distinct ends.
By Theorem \ref{thm:theorem by Brown},
we have:
\[ \mathcal{G} =
\mathcal{G}_{\{ E \cup E_1 \cup E_2  \}} *_{\mathcal{G}_{ \{ E, E_1 \cup E_2 \} }}
\mathcal{G}_{\{ E \}}.  \]
The presentation in (\ref{item:q1})-(b)
is obtained by computing each of those isotropy groups.

\smallskip

\noindent (\ref{item:q1})-(c)
By Theorem \ref{thm:structure}
the primitive disk complex $\mathcal{P}(V)$ for the genus-$2$ Heegaard splitting of $L(p,1)$, $p>3$, is a tree, which is described in Figure \ref{shape} (1)-(b).
Let $\mathcal{T}$ be the first barycentric subdivision
of $\mathcal{P}(V)$.
By Lemma \ref{lem:number of orbits of primitive pairs}
the quotient of ${\mathcal{T}}$ by the action of
$\mathcal{G}$ is a single edge with distinct ends.
By Theorem \ref{thm:theorem by Brown},
we have:
\[ \mathcal{G} =
\mathcal{G}_{\{ E \cup D  \}} *_{\mathcal{G}_{ \{ E, D \} }}
\mathcal{G}_{\{ E \}}.  \]
The presentation in (\ref{item:q1})-(c)
is obtained by computing each of those isotropy groups.

\smallskip

\noindent (\ref{item:q2}) Suppose that $q>1$.

\smallskip

\noindent (\ref{item:q2})-(a)
By Theorem \ref{thm:structure},
the primitive disk complex $\mathcal{P}(V)$ for the genus-$2$ Heegaard splitting of $L(5,2)$
is a 2-dimensional contractible complex, which is described in Figure \ref{shape} (2)-(b).
A portion of $\mathcal P(V)$ containing the vertices $E$, $D$, $E_1$ and $D_1$ is illustrated
in Figure \ref{fig:primitive_disk_complex_1} (a).
\begin{center}
\begin{overpic}[width=14cm, clip]{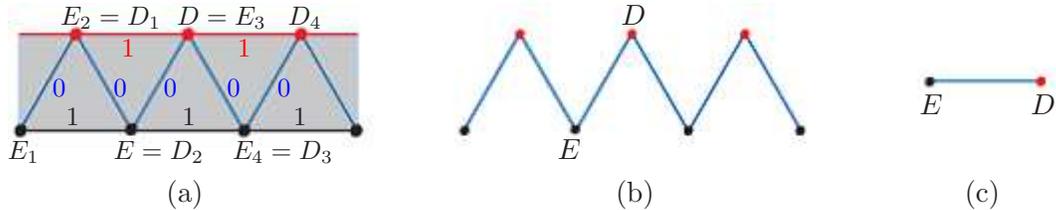}
  \linethickness{3pt}
  \put(60,0){(a)}
  \put(40,16){\small $E=D_2$}
  \put(0,16){\small $E_1$}
  \put(20,67){\small $E_2=D_1$}
  \put(64,67){\small $D=E_3$}
  \put(85,16){\small $E_4=D_3$}
  \put(106,67){\small $D_4$}
  \put(230,0){(b)}
  \put(208,16){$E$}
  \put(232,67){$D$}
  \put(362,0){(c)}
  \put(345,33){$E$}
  \put(387,33){$D$}
  \put(17,40){\color{blue} \small $0$}
  \put(40,40){\color{blue} \small $0$}
  \put(59,40){\color{blue} \small $0$}
  \put(83,40){\color{blue} \small $0$}
  \put(102,40){\color{blue} \small $0$}
  \put(22,29){\small $1$}
  \put(66,29){\small $1$}
  \put(108,29){\small $1$}
  \put(43,55){\color{red} \small $1$}
  \put(87,55){\color{red} \small $1$}
 \end{overpic}
\captionof{figure}{(a) The primitive disk complex $\mathcal{P}(V)$.
(b) The tree $\mathcal{T}$. (c) The quotient $\mathcal{T} / \mathcal{G}$.}
\label{fig:primitive_disk_complex_1}
\end{center}
We recall that
each 2-simplex of $\mathcal{P}(V)$ contains exactly two edges of type-0 (both of which are elements of
$\mathcal{G} \cdot \{ E, E_1 \}$)
and one edge of type-1 (which is an element of $\mathcal{G} \cdot \{ E, D \}$).
We observe that the subcomplex of $\mathcal P(V)$ which consists only of the type-0 edges with the vertices is a tree, which we denote by $\mathcal T$.
See Figure \ref{fig:primitive_disk_complex_1} (b).
By Lemma \ref{lem:number of orbits of primitive pairs}
the Goeritz group $\mathcal{G}$
acts without inversion on the edges of
$\mathcal{T}$ and
the two endpoints of each edge belong to different
orbits of vertices under the action of $\mathcal{G}$.
Moreover, the action is transitive on the set of the edges of $\mathcal{T}$.
Hence the quotient of ${\mathcal{T}}$ by the action of
$\mathcal{G}$ is a single edge, see
Figure \ref{fig:primitive_disk_complex_1} (c).
By Theorem \ref{thm:theorem by Brown},
we have:
\[ \mathcal{G} =
\mathcal{G}_{\{ E  \}} *_{\mathcal{G}_{ \{ E, D \} }}
\mathcal{G}_{\{ D \}}  . \]
By Lemmas \ref{lem:stabilizer of a primitive disk} and
\ref{lem:stabilizers of primitive pairs}, we get the presentation in (\ref{item:q2})-(a).

\smallskip

\noindent (\ref{item:q2})-(b)
Let $L(p,q)$ be a lens space such that
$p = 2q + 1$ and $q \geqslant 3$, or $p>5$ and $q=2$.
By Theorem \ref{thm:structure}
the primitive disk complex $\mathcal{P}(V)$ is
a 2-dimensional contractible complex, which is described in Figure \ref{shape} (2)-(c).
A portion of $\mathcal P(V)$ containing the vertices $E$, $D$, $E_1$ and $D_1$ is illustrated
in Figure \ref{fig:primitive_disk_complex_2} (a).
\begin{center}
\begin{overpic}[width=14cm, clip]{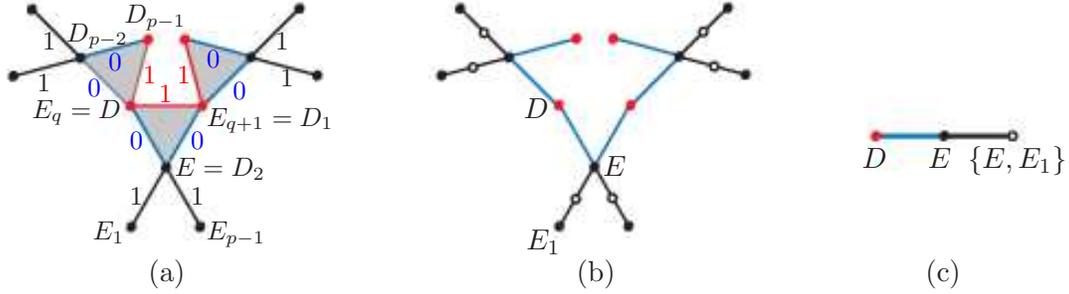}
  \linethickness{3pt}
  \put(60,0){(a)}
  \put(39,16){\small $E_1$}
  \put(70,40){\small $E = D_2$}
  \put(16,62){\small $E_q = D$}
  \put(82,59){\small $E_{q+1}=D_1$}
  \put(82,16){\small $E_{p-1}$}
  \put(29,91){\small $D_{p-2}$}
  \put(51,98){\small $D_{p-1}$}
  \put(222,0){(b)}
  \put(232,40){$E$}
  \put(202,62){$D$}
  \put(203,12){$E_1$}
  \put(354,0){(c)}
  \put(330,43){$D$}
  \put(355,43){$E$}
  \put(370,43){$\{ E , E_1 \}$}
  \put(18,72){\small $1$}
  \put(20,88){\small $1$}
  \put(110,74){\small $1$}
  \put(108,88){\small $1$}
  \put(53,29){\small $1$}
  \put(76,29){\small $1$}
  \put(53,50){\small \color{blue} $0$}
  \put(76,50){\small \color{blue} $0$}
  \put(37,70){\small \color{blue} $0$}
  \put(45,80){\small \color{blue} $0$}
  \put(92,69){\small \color{blue} $0$}
  \put(82,81){\small \color{blue} $0$}
  \put(58,75){\color{red} \small $1$}
  \put(71,75){\color{red} \small $1$}
  \put(64,68){\color{red} \small $1$}
\end{overpic}
\captionof{figure}{(a) The primitive disk complex $\mathcal{P}(V)$.
(b) The tree $\mathcal{T}'$. (c) The quotient $\mathcal{T}' / \mathcal{G}$.}
\label{fig:primitive_disk_complex_2}
\end{center}
In this case
each 2-simplex of $\mathcal{P}(V)$
contains exactly one edge of type-1
(which is an element of $\mathcal{G} \cdot \{ D, D_1 \}$)
and two edges of type-0 (both of which are elements of $\mathcal{G} \cdot \{ E, D \}$).
Substituting each $2$-simplex of $\mathcal P(V)$ by the union of the two edges of type-$0$ with their vertices in the $2$-simplex, we have a subcomplex of $\mathcal P(V)$, which is a tree.
We denote it by $\mathcal{T}$.
Let $\mathcal{T}'$ be the tree obtained from $\mathcal T'$ by adding the barycenter of each of the remaining edges of type-$1$.
See Figure \ref{fig:primitive_disk_complex_2} (b).
By Lemma \ref{lem:number of orbits of primitive pairs}
the Goeritz group $\mathcal{G}$
acts without inversion on the edges of
$\mathcal{T}'$ and
the two endpoints of each edge belong to different
orbits of vertices under the action of $\mathcal{G}$.
Moreover the complex $\mathcal{T}'$ modulo the action of $\mathcal{G}$
consists of exactly three vertices and two edges.
Hence the quotient of $\mathcal{T}'$ by the action of
$\mathcal{G}$ is the path graph on three vertices,
that is,
the tree with 3 vertices containing only vertices of degree 1 or 2.
See Figure \ref{fig:primitive_disk_complex_2} (c).
By Theorem \ref{thm:theorem by Brown},
we have
\[ \mathcal{G} =
\mathcal{G}_{\{ D  \}} *_{\mathcal{G}_{ \{ E, D \} }}
\mathcal{G}_{\{ E \}} *_{\mathcal{G}_{ \{ E , E_1 \} }}
\mathcal{G}_{\{ E \cup E_1 \}} . \]
By Lemmas \ref{lem:stabilizer of a primitive disk} and
\ref{lem:stabilizers of primitive pairs}, we obtain
the presentation in (\ref{item:q2})-(b).

\smallskip

\noindent (\ref{item:q2})-(c)
Let $L(p,q)$ be a lens space such that
$q^2 \equiv 1 \pmod p$ and $q \geqslant 3$.
By Theorem \ref{thm:structure}
the primitive disk complex $\mathcal{P}(V)$ is a tree, which is described in Figure \ref{shape} (1)-(c).
A portion of $\mathcal P(V)$ containing the vertices $E$, $D$, $E_1$ and $D_1$ is illustrated
in Figure \ref{fig:primitive_disk_complex_3} (a).
\begin{center}
\begin{overpic}[width=14cm, clip]{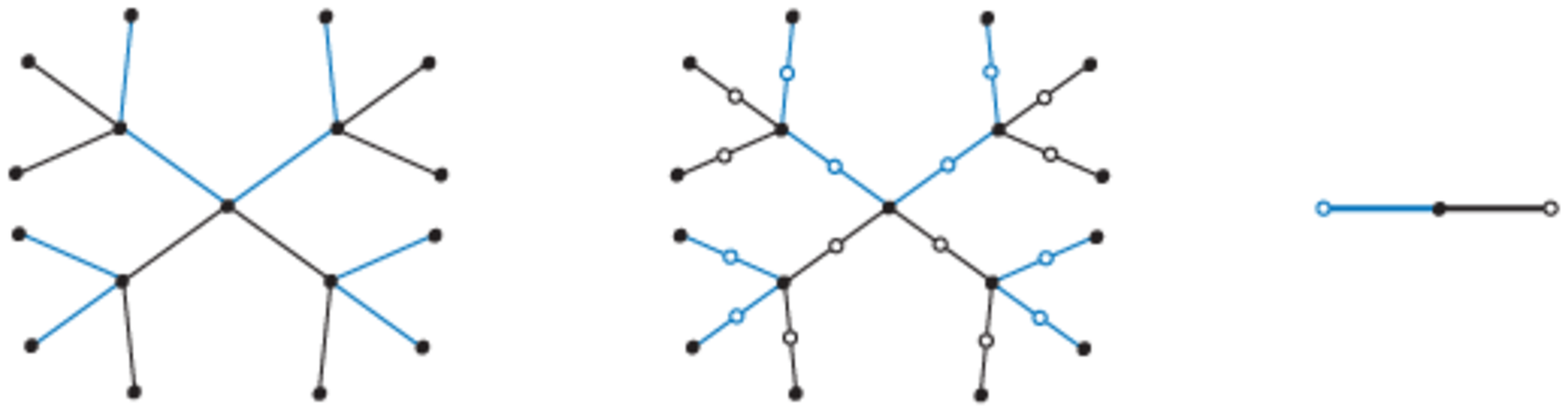}
  \linethickness{3pt}
  \put(56,0){(a)}
  \put(40,42){\small $E_1$}
  \put(-3,72){\small $D_1$}
  \put(69,64){\small $E = D_q$}
  \put(40,85){\small $D = E_q$}
  \put(65,45){\small $E_{p-1}$}
  \put(0,106){\small $D_{p-1}$}
  \put(42,111){\small $D_{p-q}$}
  \put(97,83){\small $E_{p-q}$}
  \put(218,0){(b)}
  \put(232,63){$E$}
  \put(203,85){$D$}
  \put(354,0){(c)}
  \put(315,52){$\{E, D\} $}
  \put(355,52){$E$}
  \put(373,52){$\{ E , E_1 \}$}
  \put(23,71){\small $1$}
  \put(20,87){\small $1$}
  \put(100,71){\small $1$}
  \put(97,95){\small $1$}
  \put(41,29){\small $1$}
  \put(80,29){\small $1$}
  \put(44,57){\small $1$}
  \put(78,56){\small $1$}
  \put(44,70){\small \color{blue} $0$}
  \put(79,70){\small \color{blue} $0$}
  \put(41,95){\small \color{blue} $0$}
  \put(82,95){\small \color{blue} $0$}
  \put(25,31){\small \color{blue} $0$}
  \put(20,46){\small \color{blue} $0$}
  \put(96,31){\small \color{blue} $0$}
  \put(103,46){\small \color{blue} $0$}
\end{overpic}
\captionof{figure}{(a) The primitive disk complex $\mathcal{P}(V)$.
(b) The tree $\mathcal{T}'$. (c) The quotient $\mathcal{T}' / \mathcal{G}$.}
\label{fig:primitive_disk_complex_3}
\end{center}
Let $\mathcal{T}$ be the first barycentric subdivision
of $\mathcal{P}(V)$.
See Figure \ref{fig:primitive_disk_complex_3}  (b).
By Lemma \ref{lem:number of orbits of primitive pairs}
the Goeritz group $\mathcal{G}$
acts without inversion on the edges of
$\mathcal{T}$ and
the two endpoints of each edge belong to different
orbits of vertices under the action of $\mathcal{G}$.
Moreover the complex $\mathcal{T}$ modulo the action of $\mathcal{G}$
consists of exactly three vertices and two edges.
Hence the quotient of ${\mathcal{T}}$ by the action of
$\mathcal{G}$ is the path graph on three vertices.
See Figure \ref{fig:primitive_disk_complex_3} (c).
By Theorem \ref{thm:theorem by Brown},
we have:
\[ \mathcal{G} =
\mathcal{G}_{\{ E \cup D  \}} *_{\mathcal{G}_{ \{ E, D \} }}
\mathcal{G}_{\{ E \}} *_{\mathcal{G}_{ \{ E , E_1 \} }}
\mathcal{G}_{\{ E \cup E_1 \}} . \]
By Lemmas \ref{lem:stabilizer of a primitive disk} and
\ref{lem:stabilizers of primitive pairs}, we obtain the presentation in (\ref{item:q2})-(c).

\smallskip

\noindent (\ref{item:q2})-(d)
Let $L(p,q)$ be a lens space such that
$q > 1$, $p \equiv \pm 1 \pmod q$, and
homeomorphic to none of the above.
We assume that $p \equiv 1 \pmod q$.
The argument for the case where $p \equiv -1 \pmod q$ is the same.
By Theorem \ref{thm:structure}
the primitive disk complex $\mathcal{P}(V)$ is a tree, which is described in Figure \ref{shape} (1)-(c) again.
A portion of $\mathcal P(V)$ containing the vertices $E$, $D$, $E_1$ and $D_1$ is illustrated
in Figure \ref{fig:primitive_disk_complex_4} (a).
\begin{center}
\begin{overpic}[width=14cm, clip]{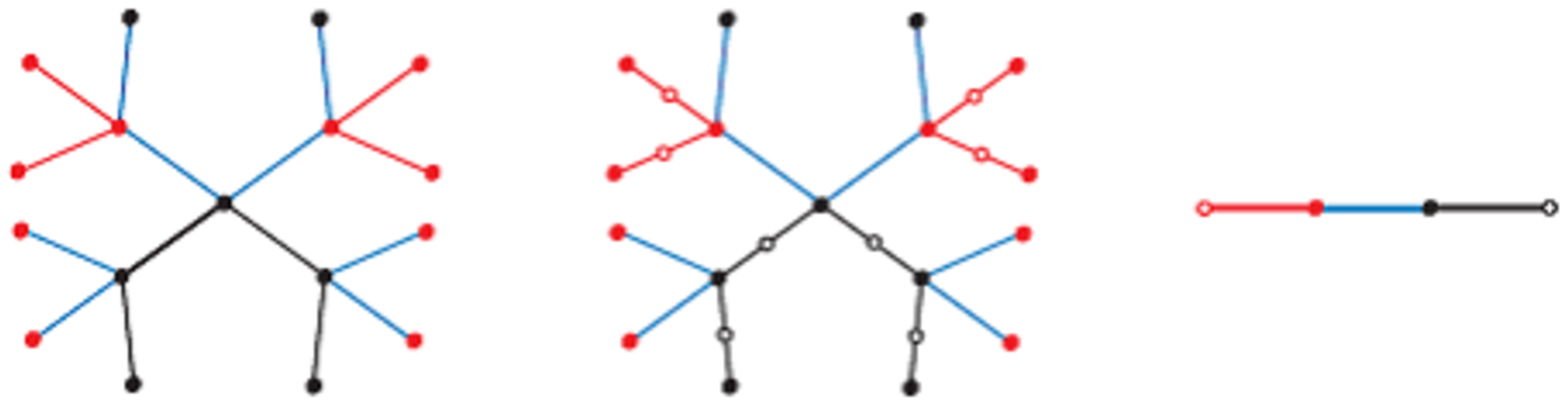}
  \linethickness{3pt}
  \put(53,0){(a)}
  \put(36,41){\small $E_1$}
  \put(-7,72){\small $D_1$}
  \put(64,63){\small $E = D_q$}
  \put(36,85){\small $D = E_{q'}$}
  \put(60,45){\small $E_{p-1}$}
  \put(-4,106){\small $D_{p-1}$}
  \put(38,111){\small $D_{p-q}$}
  \put(94,83){\small $E_{p-q'}$}
  \put(200,0){(b)}
  \put(213,62){$E$}
  \put(185,84){$D$}
  \put(339,0){(c)}
  \put(285,51){$\{ D , D_1 \}$}
  \put(326,51){$D$}
  \put(354,51){$E$}
  \put(374,51){$\{ E , E_1 \}$}
  \put(19,71){\small \color{red} $1$}
  \put(16,87){\small \color{red} $1$}
  \put(96,71){\small \color{red} $1$}
  \put(93,95){\small \color{red} $1$}
  \put(37,29){\small $1$}
  \put(76,29){\small $1$}
  \put(39,57){\small $1$}
  \put(75,56){\small $1$}
  \put(40,70){\small \color{blue} $0$}
  \put(75,70){\small \color{blue} $0$}
  \put(37,95){\small \color{blue} $0$}
  \put(78,95){\small \color{blue} $0$}
  \put(21,31){\small \color{blue} $0$}
  \put(16,46){\small \color{blue} $0$}
  \put(92,31){\small \color{blue} $0$}
  \put(99,46){\small \color{blue} $0$}
\end{overpic}
\captionof{figure}{(a) The primitive disk complex $\mathcal{P}(V)$.
(b) The tree $\mathcal{T}'$. (c) The quotient $\mathcal{T}' / \mathcal{G}$.}
\label{fig:primitive_disk_complex_4}
\end{center}
Let $\mathcal{T}$ be the tree obtained from $\mathcal P(V)$ by
adding the barycenter of each edge of type-1
(which is an element of $\mathcal{G} \cdot \{ E, E_1 \}$ or
$\mathcal{G} \cdot \{ D, D_1 \}$).
See Figure \ref{fig:primitive_disk_complex_4} (b).
By Lemma \ref{lem:number of orbits of primitive pairs}
the Goeritz group $\mathcal{G}$
acts without inversion on the edges of
$\mathcal{T}$ and
the two endpoints of each edge belong to different
orbits of vertices under the action of $\mathcal{G}$.
Moreover the complex $\mathcal{T}$ modulo the action of $\mathcal{G}$
consists of exactly four vertices and three edges.
Hence the quotient of ${\mathcal{T}}$ by the action of
$\mathcal{G}$ is the path graph on four vertices.
See Figure \ref{fig:primitive_disk_complex_4} (c).
By Theorem \ref{thm:theorem by Brown},
we have:
\[ \mathcal{G} =
\mathcal{G}_{\{ D \cup D_1  \}} *_{\mathcal{G}_{ \{ D, D_1 \} }}
\mathcal{G}_{\{ D \}} *_{\mathcal{G}_{ \{ E , D \} }}
\mathcal{G}_{\{ E \}} *_{\mathcal{G}_{ \{ E , E_1 \} }}
\mathcal{G}_{\{ E \cup E_1 \}} . \]
By Lemmas \ref{lem:stabilizer of a primitive disk} and
\ref{lem:stabilizers of primitive pairs}, we obtain the presentation in (\ref{item:q2})-(d).
\end{proof}

\smallskip
\noindent {\bf Acknowledgments.}
The authors wish to express their gratitude to Darryl McCullough for helpful discussions with his valuable advice and comments.
Part of this work was carried out while the second-named author was visiting
Universit\`a di Pisa as a
JSPS Postdoctoral Fellow for Reserch Abroad.	
He is grateful to the university and its staffs for
the warm hospitality.

\bibliographystyle{amsplain}

\end{document}